\documentclass[10pt,reqno]{amsart}
\usepackage[scale=0.75, centering, headheight=14pt]{geometry}
\usepackage[latin1]{inputenc}
\usepackage[T1]{fontenc}
\usepackage{lmodern}
\usepackage[english]{babel}

\usepackage{amsmath,amssymb,amsfonts,amsthm}
\usepackage{mathtools,accents}
\usepackage{mathrsfs}
\usepackage{xfrac}
\usepackage{array} 
\usepackage{aliascnt}
\usepackage{booktabs} 
\usepackage{array} 

\usepackage{verbatim} 
\usepackage{subfig} 

\usepackage{mathrsfs, dsfont}
\usepackage{amssymb}
\usepackage{amsthm}
\usepackage{amsmath,amsfonts,amssymb,esint}
\usepackage{graphics,color}
\usepackage{enumerate}
\usepackage{mathtools,centernot}

\usepackage{microtype}
\usepackage{paralist} 
\usepackage{cases}
\usepackage[initials]{amsrefs}
\allowdisplaybreaks

\usepackage{braket}
\usepackage{bm}

\usepackage[citecolor=blue,colorlinks]{hyperref}
\addto\extrasenglish{}

\usepackage{enumerate}
\usepackage{xcolor}
\usepackage{aliascnt}

\makeatletter
\def\newaliasedtheorem#1[#2]#3{
  \newaliascnt{#1@alt}{#2}
  \newtheorem{#1}[#1@alt]{#3}
  \expandafter\newcommand\csname #1@altname\endcsname{#3}
}
\makeatother

\usepackage{indentfirst}

\usepackage{esint}
\usepackage{mathrsfs}
\usepackage{stmaryrd}

\DeclareUnicodeCharacter{00A0}{ }
\DeclareUnicodeCharacter{00A0}{~}

\makeatletter
\newsavebox{\measure@tikzpicture}

\numberwithin{equation}{section}

\newcommand{\A}{\mathcal{A}}

\newcommand{\weak}{\overset{*}{\rightharpoonup}}

\newcommand{\R}{\mathbb{R}}
\newcommand{\Z}{\mathbb{Z}}
\newcommand{\T}{\mathbb{T}}
\newcommand{\E}{\mathds{E}}

\newcommand{\N}{\mathbb{N}}
\newcommand{\D}{\mathbb{D}}
\newcommand{\eps}{\varepsilon}

\DeclareMathOperator{\Ker}{Ker}

\DeclareMathOperator{\cof}{cof}
\DeclareMathOperator{\rank}{rank}
\DeclareMathOperator{\Id}{Id}

\DeclareMathOperator{\dv}{div}
\DeclareMathOperator{\tr}{tr}

\DeclareMathOperator{\spt}{supp}

\DeclareMathOperator{\loc}{loc}
\DeclareMathOperator{\Sym}{Sym}

\DeclareMathOperator{\dist}{dist}

\theoremstyle{plain}

\newtheorem{Teo}{Theorem}[section]
\newtheorem{QUE}[Teo]{Question}

\newtheorem{prop}[Teo]{Proposition}

\newtheorem{Cor}[Teo]{Corollary}
\theoremstyle{definition}
\newtheorem{Def}[Teo]{Definition}
\theoremstyle{remark}
\newtheorem{rem}[Teo]{Remark}

\title{Fine properties of symmetric and positive matrix fields with bounded divergence}

\author{Luigi De Rosa}
\address{L.D.R.: Department Mathematik und Informatik, Universit\"at Basel, Spiegelgasse 1,
4051 Basel, Switzerland}
\email{luigi.derosa@unibas.ch}

\author{Riccardo Tione}
\address{R.T.: Max Planck Institute for Mathematics in the Sciences, Inselstrasse 22, 04103 Leipzig, Germany}
\email{riccardo.tione@mis.mpg.de}

\begin{document}

\begin{abstract}
This paper is concerned with various fine properties of the functional 
\[
\D(A) \doteq \int_{\T^n}{\det}^\frac{1}{n-1}(A(x))\,dx
\]
introduced in \cite{SER}. This functional is defined on $X_p$, which is the cone of matrix fields $A \in L^p(\T^n;\Sym^+(n))$ with $\dv(A)$ a bounded measure. We start by correcting a mistake we noted in our \cite[Corollary 7]{USC}, which concerns the upper semicontinuity of $\D(A)$ in $X_p$. We give a proof of a refined correct statement, and we will use it to study the behaviour of $\D(A)$ when $A \in X_\frac{n}{n-1}$, which is the critical integrability for $\D(A)$. One of our main results gives an explicit bound of the measure generated by $\D(A_k)$ for a sequence of such matrix fields $\{A_k\}_k$. In particular it allows us to characterize the upper semicontinuity of $\D(A)$ in the case $A \in X_\frac{n}{n - 1}$ in terms of the measure generated by the variation of $\{\dv A_k\}_k$.  We show by explicit  example that this characterization fails in $X_p$ if $p<\frac{n}{n-1}$.  As a by-product of our characterization we also recover and generalize a result of P.-L. Lions \cite{PLL3,PLL4} on the lack of compactness in the study of Sobolev embeddings. Furthermore, in analogy with Monge-Amp\`ere theory, we give sufficient conditions under which $\det^\frac{1}{n-1}(A) \in \mathcal{H}^1(\T^n)$ when $A \in X_\frac{n}{n - 1}$, generalising the celebrated result of S. M\"uller \cite{MULDETPOS} when $A=\cof D^2\varphi$, for a convex function $\varphi$.
\end{abstract}

\maketitle

\par
\medskip\noindent
{\sc MSC (2020):  15B48, 39B42, 39B62.
\par
}

\par
\medskip\noindent

\section{Introduction}

Great interest has been given in recent years to the study of the following question. Consider a differential operator of order $k$ with constant coefficients of the form
\[
\mathcal{A} \doteq \sum_{|\alpha|= k} A_{\alpha}\partial_\alpha,\quad  \; A_\alpha \in \R^{N \times n},
\]
where
\[
\mathcal{A} : C^\infty_c(\R^m;\R^n) \to C^\infty_c(\R^m;\R^N).
\]
To every such operator we can associate the so-called \emph{wave cone}
\[
\Lambda_{\A} \doteq  \bigcup_{\xi \in \R^m\setminus\{0\}}\Ker(\mathbb{A}(\xi)) =  \{\eta \in \R^n: \exists \xi \in \R^m\setminus\{0\} \text{ s.t. } \mathbb{A} (\xi)(\eta) = 0\},
\]
where 
$$
\mathbb{A}(\xi)=(2\pi i)^k\sum_{|\alpha|=k} A_\alpha \xi^\alpha,
$$
with the usual multi-index notation $\xi^\alpha=\xi_1^{\alpha_1}\xi_2^{\alpha_2}\dots \xi_m^{\alpha_m}$. Deep results concerning measures $\mu$ satisfying $\A(\mu) = 0$ which highlight the importance of $\Lambda_\A$  were recently shown in \cite{GUIANN, RDHR}.  One of the key properties of $\Lambda_\A$ is its connection to the construction of irregular solutions to the system $\A(u) = 0$. Indeed, for every direction $\eta\in \Lambda_{\mathcal A}$, by definition there exists $\xi \in \R^m$ such that $\mathbb{A}(\xi)(\eta) = 0$. Then, given any Lipschitz profile $\varphi : \R \to \R$, the map $u(x) \doteq \varphi((x,\xi))\eta$ solves $\mathcal{A}(u) = 0$ in the sense of distributions, see for instance \cite[Sec. 4]{ST}.  Given the fact that we can choose \emph{any} profile $\varphi$, one cannot expect any kind of elliptic regularization for general solutions $u$ whose image is not disjoint from the wave cone. Thus, the question we are concerned with is whether maps $u \in L^1$ satisfying
\begin{equation}\label{distcone}
\A(u) = 0 \text{ and }\dist(u,\Lambda_\A) > 0
\end{equation}
enjoy better regularity properties, for instance $u \in L^{p}$ for $p > 1$ or at least $f(u) \in L^1$ for some superlinear function $f$. Examples of this phenomenon were found and used recently in \cite{HIGHER, GRS, CT}.  An interesting contribution to this question has been given in the seminal paper \cite{SER}, where D. Serre showed the following \emph{quasiconcavity inequality} for $A \in L^1(\T^n;\Sym^+(n))$ with $\dv A = 0$
 \begin{equation}\label{impineq}
\int_{\T^n}{\det}^\frac{1}{n-1}(A(x))\,dx \le {\det}^\frac{1}{n-1}\left(\int_{\T^n}A(x)\,dx\right),
\end{equation}
where we adopted the standard convention $(\dv A)_i=\partial_j A_{ij}$ in defining the row-wise divergence of a matrix field $A$. This result in particular implies that
\[
{\det}^\frac{1}{n-1}(A) \in L^1(\T^n).
\]
Since $\Lambda_{\dv} = \{M \in \R^{n\times n}: \rank(M) \le n -1\}$, we see that, if $\dist(A,\Lambda_{\dv}) \ge \eps|A|$, then
\[
|A|^\frac{n}{n-1} \le C(\eps){\det}^\frac{1}{n-1}(A) \in L^1,
\]
and hence \eqref{impineq} gives us another example of this improvement in integrability, which fits in the framework we described before.
\\
\\
Another consequence of the quasiconcavity inequality \eqref{impineq} is the upper semicontinuity of the functional 
\begin{equation}\label{Definition_main_functional}
\D(A) \doteq \int_{\T^n}\text{det}^\frac{1}{n-1}(A)\,dx,
\end{equation}
 as we showed in \cite{USC}.  The interest for the functional $\D$ stems from the fact that divergence-free matrix fields taking values in $\Sym^+(n)$ are ubiquitous in fluid dynamics and the calculus of variations, as explained in \cite{SER}. Its upper semicontinuity has recently been used in \cite{USC, SER20} for studying the multi-dimensional Burgers equation. From a more theoretical point of view, the quasiconcavity estimate \eqref{impineq} extends classical inequalities, such as the isoperimetric inequality and Sobolev inequalities. We show this connection in detail in Section \ref{addit}.\\
 
 Classically, instead of quasiconcavity, mathematicians studied \emph{quasiconvex} functionals, and the lower semicontinuity of the associated energy. In \cite{FM}, I. Fonseca and S. M{\"u}ller showed that energies of the form
\begin{equation}\label{energy}
\E_f(u) \doteq \int_{\T^n}f(u(x))\,dx
\end{equation} 
are weakly lower-semicontinuous on $L^q(\T^n;\R^N)\cap\ker(\mathcal{A})$, $p<q$, provided $\mathcal{A}$ satisfies Murat's constant rank condition (see \cite{FM} or \cite{MURCOM} for the definition), $f$ satisfies $|f(x)| \lesssim 1 + |x|^p$ and is $\mathcal{A}$-quasiconvex, i.e.
\begin{equation}\label{quasiconvexity_ineq}
f(A) \le \int_{\T^n}f(A + z(x))\,dx, \qquad \forall A \in \R^{N}, \forall z \in C^\infty(\T^n;\R^m) \text{ with } \mathcal{A}z = 0.
\end{equation}
For further recent developments, see \cite{AB,KR,DM,REL,WIE,GR} and references therein.
\\
\\
The higher integrability result and the upper (or lower) semicontinuity result can seem, at first, unrelated to each other, but this is not the case. Indeed, let us recall the strategy of \cite{FM} to show a weak upper semicontinuity result for $\E_f$ along a sequence $\{u_k\}_k \subset \Ker\A$ weakly converging to $u \in L^q$. Notice that, since
\[
|f(x)| \lesssim 1 + |x|^p, \quad p<q,
\] 
the sequence $\{f(u_k)\}_k$ weakly converges in $L^r$ for some $r > 1$ to a function $F$. Then the idea is, for a.e. $a \in \T^n$, to substitute the sequence $\{u_k\}_k$ with another sequence $\{v_k\}_k$ such that $v_k = u(a) + w_k$ with $w_k \rightharpoonup 0$ in $L^q$, and
\[
\int_{\T^n}f(u(a) + w_k(x))dx \to F(a).
\]
If we do so for a sequence $\{w_k\}_k\subset \Ker \mathcal{A}$, then we can use the quasiconcavity inequality, i.e. the concave counterpart of \eqref{quasiconvexity_ineq}, to obtain
\[
\int_{\T^n}f(u(a) + w_k(x))dx \le f(u(a)), \quad \forall k \in \N.
\]
Letting $k \to +\infty$, we obtain $F(a) \le f(u(a))$ for almost every $a\in\T^n$, which, upon integrating, gives us the required upper semicontinuity. In this process we have used crucially the fact that $\{f(u_k)\}_k$ converges weakly to $F$. If instead we had $q =p$, then $\{f(u_k)\}_k$ would simply be bounded in $L^1$. In this case, $\{f(u_k)\}_k$ might in principle create concentration phenomena in the limit which can prevent weak upper semicontinuity. However, as it has been recently noted in \cite{GR2}*{Thm.  4.8}, the strong convergence of $\{\mathcal{A}u_k\}_k$ prevents interior concentrations, and in turn implies weak upper semicontinuity in the critical case $p=q$. When $q<p$, this is hopeless, as we prove for instance in Proposition \ref{counter} below. More generally, if one knows that for suitable maps $\{u_k\}_k$ and $u$, for instance satisfying \eqref{distcone}, the sequence $\{f(u_k)\}_k$ does \emph{not} create concentration, then one may still wonder if upper semicontinuity of $\E_f$ still holds along $\{u_k\}_k$.
\\
\\
To summarize, we are interested in better integrability properties of maps satisfying \eqref{distcone} and weak upper semicontinuity of certain energies of the form \eqref{energy}. We can now start stating our results. Define $\Sym^+(n)$ to be the cone of  positive semi-definite matrices in the space of $n\times n$ symmetric matrices $\Sym(n)$. For $p \ge 1$, set
\[
X_p\doteq\{ A\in L^p(\T^n;\Sym^+(n)) \, : \, \dv A\in \mathcal{M}(\T^n;\R^n)\}.
\]
We will say that $A_k\rightharpoonup A$ in $X_p$ if
\[
A_k \rightharpoonup A \text{ in }L^p \text{ and } \dv(A_k) \overset{*}{\rightharpoonup} \dv(A),
\] 
where $\overset{*}{\rightharpoonup}$ denotes the usual weak-* convergence of measures as linear functionals on the space of continuous functions. See Section \ref{Sec:notations} below for the precise definition and for all the notations used here and throughout the paper. In what follows $\D=\D(\cdot)$ will always denote the functional \eqref{Definition_main_functional} defined above. As said, having the quasiconcavity inequality \eqref{impineq} means that one can hope for an upper semicontinuity result. This has been in fact shown in \cite{USC}, which we recall here.
\begin{Teo}[\cite{USC}*{Thm. 2}]\label{t:usc_supercritical}
Let $p > \frac{n}{n - 1}$ and $\{A_k\}_k\subset X_p$ be such that $A_k \rightharpoonup A$ in $X_p$. Then we have
\[
\limsup_{k\to \infty} \D (A_k)\leq \D(A).
\]
\end{Teo}
The most general version of this upper semicontinuity result is in the following:
\begin{Cor}[\cite{USC}*{Cor. 7}]\label{tbc}
Fix $p\geq 1$. Let $\{A_k\}_{k\in \N}\subset X_p$ be such that $A_k\rightharpoonup A$ in $X_p$ and assume $\{\det^{\frac{1}{n - 1}}(A_k)\}_{k \in \N}$ is equi-integrable. Then
\[
\limsup_{k \to \infty}\D(A_k) \le \D(A).
\]
\end{Cor}
Unfortunately, as we shall explain in Section \ref{error}, the proof we gave of that corollary is wrong. This paper starts by correcting it with the following result, which is indeed an improved version of Corollary \ref{tbc}.
\begin{Teo}[Characterization of the Lebesgue part]\label{introleb}
Let $\{A_k\}_{k\in \N}\subset X_1$ be a sequence such that $A_k\rightharpoonup A$ in $X_1$ and $\det ^{\frac{1}{n-1}} (A_k)$ weakly-$*$ converges to a Radon measure $\mu=h\,dx+\mu^s$ with $\mu^s \perp \mathcal{L}^n$, where $\mathcal{L}^n$ is the Lebesgue measure. Then
\begin{equation}\label{hest}
h(x) \leq \left(\det A(x)\right)^\frac{1}{n-1},\quad \text{for a.e. }x\in \T^n.
\end{equation}
In particular, if $\mu \ll \mathcal{L}^n$ we have
$$
\limsup_{k\rightarrow \infty} \mathbb{D} (A_k)\leq \mathbb{D} (A).
$$
\end{Teo}
This theorem in particular shows that Corollary \ref{tbc} holds true since equi-integrable sequences weakly converge in $L^1$ and hence do not display singular parts in the limit. Its proof is based on an adaptation of the proof of \cite[Thm. 2]{USC}.  Section \ref{uscleb} will be entirely devoted to explain the mistake in the proof of \cite[Cor. 7]{USC} and the proof of this theorem.
\\
\\
We will now focus on the main/new results of the current paper.  In what follows we will denote by $|\mu|$ the  variation measure of a vector-valued measure $\mu\in \mathcal{M}(\T^n;\R^n)$. From the theory of divergence-measure fields, see \cite{DMCT} and references therein for an overview, and in particular from \cite[Thm. 3.2]{Sil05} it is known that if $A\in X_p$, $\frac{n}{n-1}\leq p \le \infty$,  then $|\dv A|$ is absolutely continuous with respect to the Hausdorff measure $\mathcal{H}^{n-p'}$, $p'$ being the H\"older conjugate of $p$, i.e. $\frac{1}{p}+\frac{1}{p'}=1$. Note that if $p>\frac{n}{n-1}$ then $p'<n$, and hence $|\dv A|$ is diffuse. On the other hand, the counterexample to the upper semicontinuity for $p = \frac{n}{n-1}$ we proposed in \cite[Prop. 8]{USC} displays a Dirac mass in the limit of $| \dv A_k|$.   This lets us raise the following:

\begin{QUE}\label{QUES} Can better properties of $\{|\dv A_k|\}_{k}$ compensate for the lack of integrability of $\{A_k\}_k$? More precisely, denote by $\nu= w^*\text{-}\lim_{k\rightarrow \infty}|\dv A_k|$. Can the upper semicontinuity of $\D$ persist if one replaces the hypothesis $p>\frac{n}{n-1}$ in Theorem \ref{t:usc_supercritical} by a qualitative assumption $\nu \ll \mathcal{H}^\delta$, or, if necessary, by a more quantitative one $\nu \leq C \mathcal{H}^\delta $, for some positive numbers $C,\delta>0$ ?
\end{QUE}
In Proposition \ref{counter}  we show that the answer to the above question is negative in the case $p<\frac{n}{n-1}$. In particular we construct a sequence $\{A_k\}_k$ such that $\dv A_k=0$ for all $k\geq 1$, but $\limsup_k \D(A_k)>\D(A)$. The situation is more interesting for $p = \frac{n}{n-1}$, which is the natural scaling exponent of $\D(A)$. In Section \ref{charac}, we will show the following result, which gives a positive answer to Question \ref{QUES} in the critical case.
\begin{Teo}[Characterization of the singular part]\label{intro:charact_singular_part}
Let $A_k \rightharpoonup A$ in $X_{1}$. First, let $u$ be the weak $L^1(\T^n)$ limit of $\{|A_k|\}_{k \in \N}$, up to a non relabelled subsequence. We assume 
\begin{equation}\label{morereglim}
u \in L^\frac{n}{n-1}(\T^n).
\end{equation}
Let $\mu$ and $\nu$ be, respectively, the weak-star limits of the measures $\mu_k = \det^\frac{1}{n-1}(A_k)dx$ and $\nu_k = |\dv A_k|$. Denote by $\mu^s$ and $\nu^s$ the singular parts of these measures with respect to the Lebesgue measure. There exists a dimensional constant $C= C(n) > 0$ such that the following holds. If $\{x_i\}_{i\in \N}$ is the countable set  of points in $\T^n$ such that $\nu^s(\{x_i\}) > 0$, then $\mu^s$ is a discrete measure concentrated on $\{x_i\}_{i\in \N}$ and moreover it holds
\begin{equation}\label{bound_singular_concentration}
\mu^s\leq C(n) \sum_{i=1}^\infty \nu^s(\{x_i\})^\frac{n}{n-1}\delta_{x_i}\quad \text{as measures}.
\end{equation}
\end{Teo}

Note that since the measure $\nu^s$ is a finite measure, then the set of atoms $\{x_i\}_{i\in\N}$ can be at most countable.  Let us comment on our assumption \eqref{morereglim}. First, this is always satisfied if we assume $\{A_k\}_k$ is equibounded in $L^\frac{n}{n - 1}(\T^n)$. However, we preferred to give this more general statement since one can easily find examples of sequences of $\{A_k\}_k$ for which \eqref{morereglim} holds. For instance, if $\{A_k\}_k$ is an equibounded sequence in $X_1$ and weakly converges in $L^1$ to $0 \in \Sym^+(n)$, then \eqref{morereglim} holds. Indeed, since $A_k \ge 0$ for all $k$, weak convergence to $0$ implies strong convergence to $0$ in $L^1$, and hence $\lim_k|A_k| = u = 0$, which of course fulfills \eqref{morereglim}. 
\\
\\
In Section \ref{addit}, and in particular in Corollaries \ref{cor1}-\ref{cor2},  we will give a better bound than  \eqref{bound_singular_concentration}, but since it requires additional technical details we prefer not to state it here. However, an immediate by-product of Corollary \ref{cor2} is the following characterisation of compactness of the Sobolev embedding  $W^{1,p}(\T^n)\subset L^{p^*}(\T^n)$, which in turn generalizes a celebrated result of P.-L. Lions \cite{PLL3,PLL4} by relaxing the assumption on the full gradient to any directional derivative. 
\begin{Cor}\label{cor3}
Let $p \in [1,n)$ and $p^*=\frac{np}{n-p}$ be the corresponding Sobolev exponent. Consider a sequence $\{u_k\}_{k \in \N}$ which is equibounded in $W^{1,p}(\T^n)$. Assume $|u_k|^{p^*}dx \weak g\,dx + \mu^s$, with $\mu^s \perp \mathcal{L}^n$. Fix any direction $v \in \mathbb{S}^{n-1}$ and set $\gamma_k \doteq |(Du_{k},v)|^pdx$. Suppose that for some subsequence
\begin{equation}\label{dirder}
\gamma_{k_j} \weak \gamma,
\end{equation}
for some diffuse measure $\gamma$, i.e. $\gamma(\{x\}) = 0$ for all $x \in \T^n$. Then, $\mu^s \equiv 0$ and in particular, if $u_k \rightharpoonup u$ in $L^{p^*}(\T^n)$ and strongly in $L^1(\T^n)$, then the convergence is strong in $L^{p^*}(\T^n)$.
\end{Cor}
Going back to our functional $\D=\D(\cdot)$, by combining Theorems \ref{introleb}-\ref{intro:charact_singular_part}, we obtain the following complete description of possible failure of upper-semicontinuity of $\D$ in $X_{\frac{n}{n-1}}$.
 \begin{Cor}\label{intro:full_USC}
Let $\{A_k\}_{k\in \N}\subset X_1$ be such that $A_k\rightharpoonup A$ in $X_1$. Assume further that $\{|A_k|\}_k$ converges weakly to $u \in L^\frac{n}{n-1}(\T^n)$ as in \eqref{morereglim}.  Let $\nu=\text{w*-}\lim_{k\rightarrow \infty} |\dv A_k |$ and denote by $\nu^s$ its singular part with respect to the Lebesgue measure. There exists a dimensional constant $C= C(n)> 0$ such that the following statement holds true. If $\{x_i\}_{i\in \N}\subset \T^n$ is the countable set in which $\nu^s(\{x_i\})>0$, letting $\text{w\\
*-}\lim_{k\rightarrow \infty}\det^\frac{1}{n-1} A_k=\mu =h\,dx+\mu^s$ with $\mu^s\perp \mathcal{L}^n$, we have
$$
h(x)\leq \left( \det A(x)\right)^\frac{1}{n-1},\quad \text{for a.e.  } x\in \T^n
$$
and $\mu^s$ is a discrete measure concentrated on $\{x_i\}_{i\in \N}$ with the inequality
\begin{equation}\label{mu_sing_bound_nu_sing}
\mu^s\leq C(n) \sum_{i=1}^\infty \nu^s(\{x_i\})^\frac{n}{n-1}\delta_{x_i}\quad \text{as measures}.
\end{equation}
In particular, the functional $\D=\D(\cdot)$ is weakly upper semicontinuous in a point $A\in X_\frac{n}{n-1}$ along the sequences $\{A_k\}_{k\in \N}\subset X_\frac{n}{n-1}$ such that $\nu^s$ is diffuse.
\end{Cor}

This result states that the failure of upper semicontinuity of $\D(A)$ along a sequence $\{A_k\}_k \subset X_{\frac{n}{n-1}}$ is completely controlled by $\nu^s$, if $\nu$ is the weak-star limit of $\{|\dv A_k|\}_k$, and it should be compared with the aforementioned classical works on concentration-compactness of Lions \cite{PLL1,PLL2,PLL3,PLL4}. As stated,  the previous corollary implies weak upper semicontinuity of $\D$ along sequences equibounded in $L^\frac{n}{n-1}$ for which \{$| \dv A_k|_{k \in \N}$\} does not generate atoms. Moreover,  as already discussed above, we complement this result with Proposition \ref{counter}  which shows that, in the subcritical case $p<\frac{n}{n-1}$,  such a quantitative characterization fails. In particular the divergence-free sequence $\{A_k\}_k$ constructed in Proposition \ref{counter} displays an atom in $\mu^s$, i.e. the singular part of the measure generated by  $\{\det^\frac{1}{n-1} A_k\}_k$,  disproving the validity of \eqref{mu_sing_bound_nu_sing} in the case $p<\frac{n}{n-1}$ and consequently indicating the sharpness of assumption \eqref{morereglim} used in Theorem \ref{intro:charact_singular_part} and in Corollary \ref{intro:full_USC}. A similar example has been considered earlier in \cite{Serre21}*{Sec. 4}. 
\\
\\
Our final main result again concerns the interplay between $|\dv A|$ and ${\det}^\frac{1}{n-1}(A)$. Indeed, in Section \ref{hardy}, we will show by adapting a procedure devised by S. M\"uller in \cite{MULDET,MULDETPOS} the following
\begin{Teo}\label{intro:det_hardy}
Let $A \in X_{\frac{n}{n-1}}$ and suppose that
\begin{equation}\label{gencond}
\tilde M \left(|\dv A|\right) (x)\doteq\sup_{0 <R <1}\frac{1}{R^{n-1}}|\dv A|(B_R(x)\cap \T^n) \in L^\frac{n}{n-1}(\T^n).
\end{equation}
Then,
\begin{equation}\label{llogl}
\int_{\T^n}{\det}^{\frac{1}{n -1}}(A(x))\ln\left(1 + {\det}^{\frac{1}{n-1}}(A)(x)\right) \,dx \leq c\left(\|A\|_{L^{\frac{n}{n-1}}(\T^n)}, \left\|\tilde M \left(|\dv A|\right)\right\|_{L^{\frac{n}{n-1}}(\T^n)}\right).
\end{equation}
\end{Teo}

Let us put this result into context. For a convex function $\varphi\in W^{2,n}(B_1)$, where $B_1 \subset \R^n$ is the unit ball centered at $0$, we have by \cite{MULDET,MULDETPOS} the remarkable property that\footnote{It is worth noticing that the results of \cite{MULDET,MULDETPOS} imply the same inequality for general maps $u: B_1 \to \R^n$, even if they are not the gradient of a convex function. For even more general results in this direction, see \cite{DET}.}
\begin{equation}\label{lloglfi}
\int_{B_\frac{1}{2}}\det(D^2\varphi)\ln(1 + \det(D^2\varphi))dx < + \infty.
\end{equation}
In other words, $\det(D^2\varphi) \in \mathcal{H}^1_{\loc}$, the local Hardy space. Every Hessian of a $W^{2,n}$ convex function gives rise to a divergence free matrix field $A \in X_{\frac{n}{n - 1}}$ simply by setting
\[
A(x) \doteq \cof D^2\varphi(x),
\]
where $\cof$ is the cofactor matrix. Thus our Theorem \ref{intro:det_hardy} extends property \eqref{lloglfi} to a larger class of matrix-fields, which include for instance all divergence-free matrix fields in $X_{\frac{n}{n-1}}$. Furthermore, Theorem \ref{intro:det_hardy} is optimal in the following sense. One cannot hope to upgrade \eqref{llogl} to a $L^{1+\eps}$ estimate for ${\det}^\frac{1}{n-1}(A)$. Indeed, in \cite{DRT}, the authors showed that, for $p \le \frac{n}{n-1}$, there exists $A \in X_p$ such that
\[
{\det}^\frac{1}{n - 1}(A) \notin \bigcup_{\eps > 0}L^{1 + \eps}_{\loc}(\T^n),
\]
hence one cannot expect a better gain of integrability than \eqref{llogl}. 
Moreover, both the assumptions $p=\frac{n}{n-1}$ and \eqref{gencond} are essentially sharp: one cannot take $A \in X_p$ for $p < \frac{n}{n-1}$ and still hope for \eqref{llogl}, even with additional requirements on $\dv A$, and  for $p=\frac{n}{n-1}$ assumption \eqref{gencond} cannot be avoided. We give more precise details in Remarks \ref{exp} and \ref{rem_hardy_crit} at the end of the paper.
\\
\\
Note that  applying Theorem \ref{intro:det_hardy} to a sequence $\{A_k\}_k\subset X_\frac{n}{n-1}$ yields in particular the upper semicontinuity of $\D$ if  
\begin{equation}\label{tilde_M_div_bounded}
\sup_k \left\|\tilde M (|\dv A_k|)\right\|_{L^\frac{n}{n - 1}(\T^n)} \le C.
\end{equation}
However, with respect to Corollary \ref{intro:full_USC} which guarantees upper semicontinuity by the only requirement that, in the limit, the sequence of measures $\{| \dv A_k|\}_k$ does not generate atoms, Theorem \ref{intro:det_hardy} yields, assuming \eqref{tilde_M_div_bounded}, the stronger conclusion that $\{ \det^\frac{1}{n-1} A_k\}_k$ is bounded in $\mathcal{H}^1(\T^n)$. In Appendix \ref{AppB} we will show some well-known conditions which imply \eqref{gencond}. See Remark \ref{vs} for a more detailed discussion.
\\
\\
The last result we will prove in Section \ref{hardy} will be the following corollary of Theorem \ref{intro:det_hardy}.
\begin{Cor}\label{intro:hardystrong}
Let $\lambda>0$. Define the cone 
$$
C_\lambda\doteq \left\{ A\in \Sym^+(n)\, : \, \det A \geq \lambda |A|^n \right\}.
$$
Let $\{A_k\}_{k\in \N}\subset X_\frac{n}{n-1}\cap C_\lambda$ be such that $A_k \rightharpoonup A$ in $X_\frac{n}{n-1}$, $A_k\rightarrow A$ almost everywhere on $\T^n$ and the variation $|\dv A_k|$ is such that
\[
\sup_k\|\tilde M \left(|\dv A_k|\right) \|_{L^\frac{n}{n - 1}(\T^n)} \le C.
\]
Then $A_k\rightarrow A$ in $ L^\frac{n}{n-1}(\T^n)$.
\end{Cor}

Finally, in Appendix \ref{AppA}, we will recall some useful facts about Radon Measures and their Lebesgue decomposition.
\\
\\
\textbf{ Acknowledgements}. 
The second author is indebted to Jonas Hirsch for useful discussions concerning the proof of Theorem \ref{intro:charact_singular_part} and for pointing out reference \cite{DLDPHM}. We thank Denis Serre for proposing Question \ref{QUES} to us and for suggesting Corollary \ref{cor2}.  Furthermore, we thank Andr\'e Guerra and Bogdan Rai\c{t}{\u{a}} for fruitful conversations which led to Remark \ref{rem_hardy_crit} and improvements in the introduction.

\section{Notation and technical preliminaries}\label{Sec:notations}

We will denote with $\T^n$ the $n$-dimensional torus of $\R^n$, that is defined as $\R^n/\Z^n$. We identify $\T^n$ with $[0,1]^n$, so that $|\T^n| = 1$, where $|E|$ denotes the $n$-dimensional Lebesgue measure of the Borel set $E\subset \R^n$. Moreover, we see every function $f: \T^n \to \R^m$ as a $\Z^n$-periodic function defined on $\R^n$, i.e. $f(x+ z) = f(x),\forall x \in \R^n, z \in \Z^n$.
\\
\\
For a set $E \subset \R^n,\T^n$, we denote its boundary as $\partial E$ and its closure as $\overline{E}$. The standard scalar product in $\R^n$ and the Frobenius scalar product in $\R^{n\times n}$ are both denoted by $(\cdot,\cdot)$.

\subsection{Radon measures}

We denote by $\mathcal{M}(\T^n;\R^m)$ the space of bounded Radon measures with values in $\R^m$. When $m =1$, we denote this space by $\mathcal{M}(\T^n)$, and the space of positive Radon measures by $\mathcal{M}_+(\T^n)$. We recall that this is a normed space, where the norm is given by
 $$\|\mu\|_{\mathcal{M}(\T^n;\R^m)} \doteq \sup_{\Phi \in C^0(\T^n;\R^m),\|\Phi\|_{\infty}\le 1}\mu(\Phi),
 $$
with $\mu(\Phi)\doteq \int_{\T^n} \Phi\cdot d\mu $. Then the weak-star convergence of $\mu_k \in \mathcal{M}(\T^n;\R^m)$ to $\mu \in \mathcal{M}(\T^n;\R^m)$ is given by
\[
\mu_k \weak \mu  \Leftrightarrow \mu_k(\Phi) \to \mu(\Phi), \, \, \forall \Phi \in C^0(\T^n;\R^m).
\]
Since $\mathcal{M}(\T^n;\R^m)$ is the dual of $C^0(\T^n;\R^m)$ that is a separable space, we have sequential weak-$*$ compactness for equibounded sequences $\mu_k \in \mathcal{M}(\T^n;\R^m)$. See for instance \cite[Sec. 1.9]{EVG}.
\\
\\
Moreover, for a vector-valued measure $\mu$ we will denote by $|\mu|\in \mathcal{M}_+(\T^n)$ its variation, namely the positive (scalar-valued) measure defined as
$$
|\mu|(\varphi)=\int_{\T^n} \varphi \,d|\mu| \doteq \sup_{g \in C^0(\T^n;\R^m),|g|\le \varphi}\mu(g),\quad \forall \varphi\in C^0(\T^n), \, \varphi \geq 0.
$$
For every $\mu \in \mathcal{M}(\T^n;\R^m)$, we consider its Lebesgue decomposition $\mu = g\,dx + \mu^s$, where $g \in L^1(\T^n;\R^m)$ and $\mu^s \in \mathcal{M}(\T^n;\R^m)$ denotes a singular measure with respect to the Lebesgue measure. Recall that, if $\alpha,\beta \in \mathcal{M}(\T^n;\R^m)$, then $\alpha$ is said to be singular with respect to $\beta$, or simply $\alpha \perp \beta$, if there exists $A \subset \T^n$ with $|\beta|(A) = 0$ and
\[
|\alpha|(E) = |\alpha|(E\cap A), \quad  \text{for every Borel set } E \subset \T^n.
\]
We will recall more precise facts about decomposition of measures in Section \ref{AppA}. A Lebesgue point for a function $g \in L^1(\T^n;\R^m)$ is a point $x$ such that
$$
\lim_{r\rightarrow 0^+}\fint_{B_r(x)}\left|g(y) - g(x)\right|\,dy=0, \quad \text{ where } \quad \fint_{E} f(y)\,dy = \frac{1}{|E|}\int_E f(y)\,dy,$$ for every $f \in L^1(\R^n)$, $ E$ Borel subset of $\R^n$  with $|E|>0$. It is well known that the set of Lebesgue points of such a function $g$ are of full measure in $\R^n$ (see for instance \cite[Theorem 1.33]{EVG}). More generally, if $\mu \in \mathcal{M}_+(\T^n)$ or $\mathcal{M}_+(\R^n)$, we call its (upper) density the function
\[
D\mu(x) \doteq \limsup_{r \to 0^+}\frac{\mu(\overline{B_r(x))}}{\omega_nr^n},
\]
where $\omega_n \doteq |B_1(0)|$ is the $n$-dimensional Lebesgue measure of the unit ball. We will use the fact that, if $\mu$ is singular with respect to the Lebesgue measure, then $D\mu(x) = 0$ for a.e. point of $\T^n$, see \cite[Thm. 1.31]{EVG}.

\subsection{Weak compactness criterion in $L^1$} We recall that a sequence $\{f_k\}_{k \in \N}$ is equi-integrable if for every $\varepsilon > 0$ there exists $\delta > 0$ such that if $|E| \le \delta$, then
\[
\sup_{k \in \N}\int_E|f_k(x)|dx \le \varepsilon.
\]
The importance of equi-integrability stems from the fact that a bounded sequence of $L^1(\T^n)$ functions $\{f_k\}_{k \in \N}$ is weakly precompact in $L^1(\T^n)$ if and only if it is equi-integrable, see \cite[Thm. 4.30]{BRE}.

\subsection{Linear algebra facts} For symmetric matrices $A,B \in \Sym^+(n)$, we use the standard notation
$
A \ge B
$
to denote the partial order relation
\[
(Av,v) \ge (Bv,v),\quad  \forall v \in \R^n.
\]
Recall the basic monotonicity property of the determinant on $\Sym^+(n)$:
\[
A \ge B \ge 0 \Rightarrow \det(A) \ge \det(B).
\]
For a matrix $A$, we denote with $P_A(\lambda)$ its characteristic polynomial, i.e. $P_A(\lambda) \doteq \det(\lambda\Id - A)$. Let us define, for a matrix $A \in \Sym^+(n)$ with eigenvalues $\lambda_1,\dots,\lambda_n$, the sum of the $i\times i$ minors
\begin{equation}\label{M}
M_i(A) \doteq \sum_{1\le j_1\le\dots\le j_i \le n}\lambda_{j_1}\dots\lambda_{j_i}, \quad \forall i \in \{1,\dots, n\},\; M_0(A) \doteq 1.
\end{equation}
It is a basic Linear Algebra fact that, for  $0 \le i \le n$ the $i$-th coefficient of $P_A(\lambda)$ is given by $(-1)^{i + n}M_{n - i}(A)$. Notice in particular that $M_n(A) = \det(A)$.

\section{Upper semicontinuity of the Lebesgue part}\label{uscleb}

The main part of this section will be the proof of Theorem \ref{introleb}, but we start by explaining our mistake in the proof of Corollary \ref{tbc}.

\subsection{Mistake in the original proof}\label{error} The main mistake in the original proof of \cite[Cor. 7]{USC} stands in the claim that there exists a constant $C=C(n,\varepsilon)>0$ such that 
\begin{equation}\label{wrong_inequality}
|B|^n\leq C \det B,
\end{equation} 
for every $B\in \Sym^+(n)$ with $B\geq \varepsilon \Id_n$.  Indeed by taking in the $2$-dimensional case $n=2$ the sequence 
$$
B_N = 
\begin{pmatrix}
N & 0 \\
0 & 1
\end{pmatrix}
$$
one gets $\|B_N\|^2=N^2+1$ while $\det B_N=N$. By letting $N\rightarrow \infty$, this shows that \eqref{wrong_inequality} cannot hold with a constant $C>0$ which does not depend on the matrix $B$.

\subsection{A technical preliminary to Theorem \ref{introleb}}\label{pos}

 We remark that it is sufficient to prove the theorem in the case in which $A_k, A\geq \varepsilon \Id_n$ for some $\varepsilon>0$. Assume indeed that the statement holds true for any sequence $\{B_k\}_{k \in \N}$ with $B_k \ge \eps \Id_n$ for some $\eps> 0$. Given a sequence $\{A_k\}_{k \in \N}$ with $A_k \rightharpoonup A$ in $X_1$, we can set $A_{k}^\eps \doteq A_k + \eps\Id_n$ for $\eps > 0$ and all $k \in \N$, for which one proved the validity of 
 \begin{equation}\label{claim_holds_for positive}
h^\eps (x)\leq \left(\det A^\eps(x)\right)^\frac{1}{n-1},
  \end{equation}
 $h^\eps$ being the absolutely continuous part of the measure $\mu^\eps=w{^*}\text{-}\lim_{k\rightarrow \infty}\text{det}^{\frac{1}{n-1}} (A^\eps_k)$.
By monotonicity of the determinant on the cone of positive definite matrices, we have
$$
\left(\det A_k^\eps\right)^\frac{1}{n-1}\ge\left(\det A_k\right)^\frac{1}{n-1}, \quad \forall k\geq 1,
$$
which, in the limit $k\rightarrow \infty$, implies that $\mu^\eps \geq \mu$ as measures. This implies $h^\eps(x) \geq h(x)$ for almost every $x\in \T^n$, and consequently by \eqref{claim_holds_for positive}
$$
h(x)\leq h^\eps (x) \leq  \left(\det A^\eps(x)\right)^\frac{1}{n-1}.
$$
Thus the corollary in the general case follows by letting $\varepsilon\rightarrow 0$.

\subsection{Proof of Theorem \ref{introleb}}
The proof we propose is an adaptation of the main theorem of \cite{USC}.
\\
\\
By Subsection \ref{pos} we can suppose that $A_k, A\geq \varepsilon \Id_n$ for some $\varepsilon>0$ and for all $k \in \N$.  Moreover, denote by $h:\T^n\rightarrow \R$ the density of the measure $\mu$ with respect to $\mathcal{L}^n$ from the statement of the corollary, that is 
\begin{equation}\label{h_density}
\mu =h\,dx+\mu^s,
\end{equation}
for some $\mu^s\perp \mathcal{L}^n$.
\\
\\
\indent\fbox{\textbf{Step 1:} definition of the main objects}
\\
\\
Let $\nu_{k} \in \mathcal{M}_+(\T^n)$ be the finite Radon measures defined by $\nu_k (E)= |\dv(A_k)|(E)$ and call $\nu$ its weak-$*$ limit, that we can always suppose to exist up to further subsequences. Recall the definition of $M_{n-i}$ given in $\eqref{M}$ and note that:
\begin{equation}\label{minorsest}
|M_{n-i}(B)|\leq C(n)|B|^{n-i}, \quad \forall B \in \Sym(n).
\end{equation}
Since $A_k \rightharpoonup A$ weakly in $L^1$, the Dunford-Pettis weak compactness criterion in $L^1$ implies that $\{A_k\}_k$ is an equi-integrable sequence. This observation in conjunction with \eqref{minorsest} shows that the same holds for 
\[
\left\{M^\frac{1}{n-1}_{n-i}(A_k(x))\right\}_{k \in \N},
\]
 as soon as $i \neq 0$. Thus, we obtain that $\forall i\neq 0$ there exists $m_{n-i}\in L^1$ such that 
\begin{equation}\label{minors_weak_compact}
M_{n-i}^{\frac{1}{n-1}}(A_k)\rightharpoonup m_{n-i}\quad \text{in } L^1.
\end{equation}
Consider $T' \subset \T^n$ to be the set of points $a \in \T^n$ such that
\begin{itemize}
\item $a$ is a Lebesgue point for $x \mapsto A(x)$ and $|A(a)|<\infty$;
\item $a$ is a Lebesgue point for $x \mapsto h(x)$ in \eqref{h_density} and $h(a)<\infty$;
\item $a$ is a Lebesgue point for every function $x \mapsto m_{n-i}(x)$, $i\neq 0$ and $m_{n-i}(a)<\infty$;
\item $a$ is a density zero point for $\mu^s$.
\end{itemize}
Since these are  $L^1(\T^n)$ functions and $\mu^s\perp \mathcal{L}^n$, we get $\mathcal{L}^n(\T^n\setminus T') = 0$.  Let $\nu = g\,dx + \nu^s$ be the Lebesgue decomposition of the weak-$*$ limit of $\nu_k$, and define $T'' \subset \T^n$ to be the set of points that are both Lebesgue points for $g$ and density $0$ points for $\nu^s$. By \cite[Thm. 1.31]{EVG}, $\mathcal{L}^n(\T^n\setminus T'') =0$. Finally, define $T \doteq T'\cap T''\cap \T^n$. We want to prove
\begin{equation}\label{FM_ineq}
h(a) \le \det(A(a))^{\frac{1}{n - 1}}, \quad \forall a \in T.
\end{equation}
Therefore, from now on we fix $a \in T$. Consider a cut-off function  $\varphi \in C^\infty_c((0,1)^n)$, $0\leq \varphi\leq 1$. For $k \in \N$ and $R>0$, we define $B_{k,R}$ over $(0,1)^n$ by
\[
B_{k,R}(x)\doteq \varphi(x) A_{k}(a+Rx) + (1 - \varphi(x))A(a).
\]
Moreover, define, for $\rho_\eta$ the standard family of nonnegative, smooth mollifiers,
\[
B_{k,R,\eta}(x) \doteq  \varphi(x) (A_{k}\star \rho_\eta)(a+Rx) + (1 - \varphi(x))A(a),
\]
for $\eta>0$ sufficiently small in terms of  $\varphi$. Remark that $B_{k,R}, B_{k,R,\eta}\equiv A(a)$ near the boundary of $[0,1]^n$, therefore they can be extended by periodicity to $\R^n$. Notice moreover that $B_{k,R}$ and $B_{k,R,\eta}$ take values in $\Sym^+(n)$.
\\
\\
\indent\fbox{\textbf{Step 2:} Monge-Amp\`ere and the main inequality}
\\
\\
First, we need to apply \cite[Thm. 2.2]{YAN}. This asserts that, for every $S \in \Sym^+(n)$, and for every smooth positive function $f:\T^n \to \R^+$, there exists a solution $\phi \in C^\infty(\T^n)$ of the Monge-Amp\`ere-type equation
\[
\det(D^2\phi(x) + S) = f(x), \forall x \in \T^n,
\]
provided $f$ satisfies
\[
\int_{\T^n}f(x)\,dx = \det(S).
\]
Furthermore, $D^2\phi(x) + S \in \Sym^+(n), \forall x \in \T^n$ and $\phi$ is uniquely determined up to constants. Therefore, we set
\[
f_{k,R}\doteq \det(B_{k,R})^{\frac{1}{n - 1}} \text{ and } f_{k,R,\delta}\doteq \det(B_{k,R})^{\frac{1}{n - 1}}\star \rho_\delta=f_{k,R}\star \rho_\delta.
\]
We let $\phi_{k,R,\delta}: \T^n \to \R$ be the solution of 
\begin{equation}\label{eq_c}
\det(D^2\phi_{k,R,\delta} + S_{k,R,\delta}) = f_{k,R,\delta},
\end{equation}
where $D^2\phi_{k,R,\delta}(x) + S_{k,R,\delta} \in \Sym^+(n), \forall x \in \T^n$. The precise form of the matrix $S_{k,R,\delta}$ will be given later, but in order to apply the previous result we need to impose the constraint
\begin{equation}\label{con}
\det(S_{k,R,\delta}) = \int_{\T^n}f_{k,R,\delta}(x)\, \,dx = \int_{\T^n}f_{k,R}(x)\, \,dx.
\end{equation}
In the last computation, we used the fact that mollification is an average-preserving operation on $\T^n$ as a simple consequence of Fubini's Theorem. Note that \eqref{eq_c} is equivalent to 
\begin{equation}\label{eq_c_2}
\det(D^2\psi_{k,R,\delta} ) =  f_{k,R,\delta},
\end{equation}
where $D^2\psi_{k,R,\delta}(x)$ is positive definite $\forall x \in \T^n$ and $\psi_{k,R,\delta}(x)\doteq \frac{1}{2}x^T S_{k,R,\delta} x +\phi_{k,R,\delta}(x)$. We will assume that
\begin{equation}\label{zer}
\phi_{k,R,\delta}(a) = 0,\quad  \forall k,R, \delta
\end{equation}
since the solution of \eqref{eq_c_2} is uniquely determined up to constants. We have, for all $k,R,\delta, \eta$,
$$
g_{k,R,\delta,\eta}\doteq \left(f_{k,R,\delta}\det(B_{k,R,\eta})\right)^\frac{1}{n}=\left( \det(D^2\psi_{k,R,\delta} B_{k,R,\eta} ) \right)^\frac{1}{n}.
$$
Since, for every $x \in \T^n$, $k \in \N$, $R,\eta,\delta > 0$, $D^2\psi_{k,R,\delta}(x)B_{k,R,\eta}(x)$ is the product of two symmetric and positive definite matrices, their product is diagonalizable with positive eigenvalues, see \cite[Prop. 6.1]{SERBOOK}. Dropping the dependence of $k,R,x,\delta,\eta$, if we call these eigenvalues $\lambda_1,\dots, \lambda_n$ we can write
\[
g_{k,R,\delta,\eta} = \left( \det(D^2\psi_{k,R,\delta} B_{k,R,\eta} ) \right)^\frac{1}{n} = (\lambda_1\dots\lambda_n)^{\frac{1}{n}} \le \frac{\sum_{i = 1}^n\lambda_i}{n},
\]
where in the last inequality we used the arithmetic-geometric mean inequality. Hence,
\[
g_{k,R,\delta,\eta} \le \frac{\tr(D^2\psi_{k,R,\delta} B_{k,R,\eta})}{n}.
\]
We rewrite for every $x \in \T^n$:
\[
\tr(D^2\phi_{k,R,\delta}B_{k,R,\eta})  = \dv(B_{k,R,\eta}D \phi_{k,R,\delta}) - (\dv(B_{k,R,\eta}),D\phi_{k,R,\delta}),
\]
from which we finally get, using the definition of $\psi_{k,R,\delta}$,
\begin{equation}\label{tbin}
g_{k,R,\delta,\eta}\le\frac{1}{n} (\tr(B_{k,R,\eta} S_{k,R,\delta} )+\dv(B_{k,R,\eta}D \phi_{k,R,\delta}) - (\dv(B_{k,R,\eta}),D\phi_{k,R,\delta}).
\end{equation}
We consider $S_{k,R,\delta}$ of the form
\begin{equation}\label{SkR}
S_{k,R,\delta} = \lambda_{k,R,\delta}\cof\left(\int_{\T^n}B_{k,R}(x)\,dx\right),
\end{equation}
for some real number $\lambda_{k,R,\delta}$ to be determined in order to fulfill \eqref{con}. Actually \eqref{con}-\eqref{SkR} imply that $\lambda_{k,R,\delta}$ and hence $S_{k,R,\delta}$, do not depend on $\delta >0$, hence we will simply write $\lambda_{k,R}$ and $S_{k,R}$ from now on. In other words, we must have
\[
\int_{\T^n}\det(B_{k,R})^{\frac{1}{n - 1}}(x)\,dx \overset{\eqref{con}}{=} \det(S_{k,R}) \overset{\eqref{SkR}}{=}\det\left(\lambda_{k,R}\cof\left(\int_{\T^n}B_{k,R}(x)\,dx\right)\right)
\]
Using the fact that $\det(\cof(X)) = \det(X)^{n - 1}$ for every $X \in \R^{n\times n}$, we solve the previous equation for $\lambda_{k,R}$ and obtain
\begin{equation}\label{lam}
\displaystyle\lambda_{k,R}=\frac{\left(\int_{\T^n}\det(B_{k,R})^{\frac{1}{ n - 1}}(x)\,dx\right)^{\frac{1}{n}}}{\left(\det\left(\int_{\T^n}B_{k,R}(x)\,dx\right)\right)^{\frac{n - 1}{n}}}.
\end{equation}
Notice that we could divide by the term 
\[
\left(\det\left(\int_{\T^n}B_{k,R}(x)\,dx\right)\right)^{\frac{n - 1}{n}}
\]
since $B_{k,R} \ge \eps \Id_n$ a.e. in $\T^n$. Observing that $\int_{\T^n}\dv(B_{k,R,\eta}D \phi_{k,R,\delta})\,dx = 0$, we integrate \eqref{tbin} to get
\begin{equation}\label{equaz1}
\int_{\T^n}g_{k,R,\delta,\eta}(x)\,dx \le \frac{1}{n}\int_{\T^n}\tr(B_{k,R,\eta}S_{k,R})\,dx- \frac{1}{n}\int_{\T^n}(\dv(B_{k,R,\eta}),D\phi_{k,R,\delta})\,dx.
\end{equation}
We will still need to manipulate this equation by bounding terms on the right hand side, and we will then let $\delta \to 0^+$ and $\eta \to 0^+$. To this aim, we start by noticing that, since $B_{k,R,\eta}$ converges strongly in $L^1$ to $B_{k,R}$ for every fixed $k$ and $R$, we see that
\begin{equation}\label{limiteta}
\lim_{\eta \to 0^+}\frac{1}{n}\int_{\T^n}\tr(B_{k,R,\eta}S_{k,R})\,dx = \left(\det\left(\int_{\T^n}B_{k,R}(x)\,dx\right)\right)^{\frac{1}{n}}\left(\int_{\T^n}\det(B_{k,R})^{\frac{1}{ n - 1}}(x)\,dx\right)^{\frac{1}{n}}.
\end{equation}
Moreover by Fatou's Lemma
\begin{equation}\label{Fat1}
\int_{\T^n}f_{k,R}(x) \,dx \le \liminf_{\eta \to 0}\liminf_{\delta \to 0}\int_{\T^n}g_{k,R,\delta,\eta}(x)\,dx.
\end{equation}
Estimating the right hand side of \eqref{equaz1} is more complicated. We start by observing as in \cite[Sec. 5.2]{SER} that $\psi_{k,R,\delta}$ is convex  for every $k,R,\delta$, and moreover the estimate
\begin{equation}\label{comp}
\|D \phi_{k,R,\delta}\|_{L^\infty(\T^n)} \le \gamma|S_{k,R,\delta}| = \gamma |S_{k,R} |
\end{equation}
holds for every $k \in \N$, $R >0$ and $\delta > 0$ for some $\gamma = \gamma(n)$.  The constraint \eqref{zer} combined with \eqref{comp} actually shows
\begin{equation}\label{comp1}
\|\phi_{k,R,\delta}\|_{L^\infty(\T^n)} +\|D \phi_{k,R,\delta}\|_{L^\infty(\T^n)} \le \gamma |S_{k,R,\delta} | = \gamma |S_{k,R}|.
\end{equation}
for a possibly larger constant $\gamma = \gamma(n)$. We claim the following fact, which we will show at the end of the proof:
\begin{equation}\label{imp}
L(a) \doteq \limsup_{R\to 0^+}\limsup_{k \to +\infty} |S_{k,R} | < + \infty.
\end{equation}
Now we rewrite the second term in the right hand side of \eqref{equaz1} as follows. First note that everywhere on $\T^n$
\[
\dv(B_{k,R,\eta}) = \varphi(x)R\dv(A_k\star \rho_\eta)(a + Rx) + [(A_k\star \rho_\eta)(a + Rx) -A(a)]D\varphi(x).
\]
Therefore
\begin{align*}
\int_{\T^n}(\dv(B_{k,R,\eta}),D\phi_{k,R,\delta})\,dx &= R\int_{\T^n}\varphi(x)(\dv(A_k \star \rho_\eta)(a + Rx),D\phi_{k,R,\delta})\,dx \\
&+ \int_{\T^n}(((A_k\star \rho_\eta)(a + Rx) -A(a))D\varphi,D\phi_{k,R,\delta})\,dx.
\end{align*}
Using the divergence theorem, we can write the last term as:
\begin{align*}
\int_{\T^n}(((A_k\star \rho_\eta)(a + &Rx)  -A(a))D\varphi,D\phi_{k,R,\delta})\,dx \\
&=-R\int_{\T^n}(\dv (A_k\star \rho_\eta)(a + Rx), D\varphi)\phi_{k,R,\delta}\,dx \\
&\quad-\int_{\T^n}((A_k\star \rho_\eta)(a + Rx) -A(a),D^2\varphi) \phi_{k,R,\delta}\,dx.
\end{align*}
Summarizing, we have
\begin{align*}
\int_{\T^n}(\dv(B_{k,R,\eta}),D\phi_{k,R,\delta})\,dx &= R\int_{\T^n}\varphi(x)(\dv(A_k\star \rho_\eta)(a + Rx),D\phi_{k,R,\delta})\,dx   \\
&\quad-R\int_{\T^n}(\dv (A_k\star \rho_\eta)(a + Rx), D\varphi)\phi_{k,R,\delta}\,dx\\
&\quad-\int_{\T^n}((A_k\star \rho_\eta)(a + Rx)-A(a),D^2\varphi) \phi_{k,R,\delta}\,dx.
\end{align*}
Using \eqref{comp1}-\eqref{imp}, we find that, for some $C > 0$ depending on $L(a)$:
\begin{equation}\label{BkR3}
\begin{split}
\left|\int_{\T^n}(\dv(B_{k,R,\eta}),D\phi_{k,R,\delta})\,dx\right| &\le CR\int_{\T^n}\Big(|\varphi(x)| + |D\varphi(x)|\Big)|\dv(A_k\star \rho_\eta)|(a + Rx)\,dx   \\
&\quad+\left|\int_{\T^n}\Big((A_k\star \rho_\eta)(a + Rx) -A(a),D^2\varphi\Big) \phi_{k,R,\delta}\,dx\right|.
\end{split}
\end{equation}
In our computations, we will always denote constants by $C$, which may vary line by line, and which may depend on every fixed quantity in this proof, for instance $\varepsilon$ from Subsection \ref{pos} and $a \in T$, but never on $\eta,\delta,k,R$. Now \eqref{comp1} in conjunction with Ascoli-Arzel\`a compactness criterion allows us, for every fixed $k,R$, to pick a sequence $\{\delta_j\}_j$ with $\delta_j \to 0$ such that
\begin{equation}\label{C0conv}
\phi_{k,R,\delta_j} \to \phi_{k,R} \text{ in }C^0(\T^n).
\end{equation}
Notice that in the limit we still have
\begin{equation}\label{comp2}
\|\phi_{k,R}\|_{L^\infty(\T^n)} +\|D \phi_{k,R}\|_{L^\infty(\T^n)} \le \gamma|S_{k,R}|.
\end{equation}
Now, denoting $Q_R(a)=a+ [0,R]^n$:
\begin{align*}
\int_{\T^n}\Big(|\varphi(x)| &+ |D\varphi(x)|\Big)|\dv(A_k\star \rho_\eta)|(a + Rx)\,dx \\
&= \frac{1}{R^n}\int_{Q_R(a)}\left(\left|\varphi\left(\frac{x-a}{R}\right)\right| + \left|D\varphi\left(\frac{x-a}{R}\right)\right|\right)|\dv((A_k\star \rho_\eta))|(x)\,dx\\
& \le \frac{C|\dv A_k \star \rho_\eta|( Q_R(a))}{R^n}.
\end{align*}
By \cite[Thm.  4.36]{MAG}, we see that
\[
|\dv(A_k\star \rho_\eta)| = |\dv(A_k)\star \rho_\eta|\weak |\dv A_k| = \nu_k, \text{ as }\eta \to 0. 
\]
Hence, by weak-$*$ convergence of measures and since $Q_{R}(a)$ is a compact set we have, see \cite[Thm. 1.40]{EVG},
\begin{equation}\label{dveta}
\limsup_{\eta \to 0^+}\int_{\T^n}\Big(|\varphi(x)| + |D\varphi(x)|\Big)|\dv (A_k\star \rho_\eta)|(a + Rx)\,dx \le \frac{C|\dv A_k|( Q_R(a))}{R^n}.
\end{equation}
Moreover, due to the strong convergence of $A_k\star \rho_\eta$ towards $A_k$ and \eqref{C0conv}, we compute
\begin{equation}\label{easyterm}
\begin{split}
\lim_{\eta\to 0}&\lim_{j\to\infty}\left|\int_{\T^n}((A_k\star \rho_\eta)(a + Rx) -A(a),D^2\varphi) \phi_{k,R,\delta_j}\,dx\right| \\
&=\lim_{j\to \infty}\lim_{\eta\to 0} \left|\int_{\T^n}((A_k\star \rho_\eta)(a + Rx) -A(a),D^2\varphi) \phi_{k,R,\delta_j}\,dx\right| \\
&=\left|\int_{\T^n}(A_k(a + Rx) -A(a),D^2\varphi) \phi_{k,R}\,dx\right|
\end{split}
\end{equation}
Now combining \eqref{limiteta}-\eqref{Fat1}-\eqref{BkR3}-\eqref{dveta}-\eqref{easyterm}, \eqref{equaz1} yields for all $k \in \N$ and $R > 0$:
\begin{equation}\label{equaz2}
\begin{split}
\int_{\T^n}{\det}^\frac{1}{n-1}(B_{k,R})\,dx &\le \left(\det\left(\int_{\T^n}B_{k,R}\,dx\right)\right)^{\frac{1}{n}}\left(\int_{\T^n}\det(B_{k,R})^{\frac{1}{ n - 1}}\,dx\right)^{\frac{1}{n}}+\frac{C|\dv A_k|( Q_R(a))}{R^{n-1}}\\
& \quad +C\left|\int_{\T^n}(A_k(a + Rx) -A(a),D^2\varphi) \phi_{k,R}\,dx\right|.
\end{split}
\end{equation}
Define
\[
\gamma_{k,R} \doteq \left(\int_{\T^n}\det(B_{k,R})^{\frac{1}{ n - 1}}(x)\,dx\right)^{\frac{1}{n}}.
\]
By the monotonicity of the determinant and the fact that $A_k(x) \ge \varepsilon \Id_n,\forall x \in \T^n, \forall k \in \N$, and $A(a) \ge \varepsilon\Id_n$, we have $B_{k,R} \ge \varepsilon\Id_n, \forall k,R$, that implies
\begin{equation}\label{bound}
\gamma_{k,R} \ge  \varepsilon^{\frac{1}{n - 1}},\quad  \forall k,R.
\end{equation}
Dividing \eqref{equaz2} by $\gamma_{k,R}$ and using \eqref{bound}, we can estimate for a constant $C >0$:
\begin{equation*}
\begin{split}
\left(\int_{\T^n}{\det}^\frac{1}{n-1}(B_{k,R})\,dx\right)^\frac{n - 1}{n} &\le \left(\det\left(\int_{\T^n}B_{k,R}(x)\,dx\right)\right)^{\frac{1}{n}}+\frac{C|\dv A_k|( Q_R(a))}{R^{n-1}}\\
& \quad +C\left|\int_{\T^n}(A_k(a + Rx) -A(a),D^2\varphi) \phi_{k,R}\,dx\right|.
\end{split}
\end{equation*}
By monotonicity of the determinant, we can further bound from below:
\begin{equation}\label{Mon1}
\int_{\T^n} \varphi(x)^{\frac{n}{n-1}}\det(A_{k}(a+Rx))^\frac{1}{n - 1}\,dx \le \int_{\T^n}{\det}^\frac{1}{n-1}(B_{k,R})\,dx.
\end{equation}
Hence we can write \eqref{equaz2} in \textcolor{orange}{its} final form:
\begin{equation}\label{equazfin}
\begin{split}
\left(\int_{\T^n}\varphi(x)^{\frac{n}{n-1}}\det(A_{k}(a+Rx))^\frac{1}{n - 1}\,dx\right)^\frac{n - 1}{n} &\le \left(\det\left(\int_{\T^n}B_{k,R}(x)\,dx\right)\right)^{\frac{1}{n}} +\frac{C|\dv A_k|( Q_R(a))}{R^{n-1}}\\
& \quad +C\left|\int_{\T^n}(A_k(a + Rx) -A(a),D^2\varphi) \phi_{k,R}\,dx\right|.
\end{split}
\end{equation}
Now we can use \eqref{comp2} in conjunction with our claim \eqref{imp} to see, through a diagonal argument, that there exists a subsequence $k_j$ independent of $m$ such that  $\phi_{k_j,\frac{1}{m}}$ converges uniformly to a continuous function $\phi_{\frac{1}{m}}$ as $j \to \infty$ for every $m \in \N$ fixed. This is immediate once we observe that \eqref{comp2} and \eqref{imp} imply that, for every fixed $m\in \N$,  $\{\phi_{k,\frac{1}{m}}\}_{k \in \N}$ is a family of equi-Lipschitz functions on $\T^n$. Moreover we find a constant $\lambda > 0$ such that
\begin{equation}\label{unif}
\|\phi_{\frac{1}{m}}\|_{C^0(\T^n)} \le \lambda, \quad \forall m \in \N,
\end{equation}
which is again immediate from \eqref{comp2}-\eqref{imp} and uniform convergence. We let
\begin{align*}
I_{j,m} &\doteq \int_{\T^n}\varphi(x)^{\frac{n}{n-1}}\det\left(A_{k_j}\left(a+\frac{x}{m}\right)\right)^\frac{1}{n - 1}\,dx,\\
II_{j,m} &\doteq \det\left(\int_{\T^n}B_{k_j,\frac{1}{m}}(x)\,dx\right),\\
III_{j,m} &\doteq  m^{n-1}|\dv A_{k_j}|(Q_\frac{1}{m}(a)),\\
IV_{j,m} &\doteq \left|\int_{\T^n}\left(A_{k_j}\left(a + \frac{x}{m}\right) -A(a),D^2\varphi\right) \phi_{k_j,\frac{1}{m}}\,dx\right|.
\end{align*}
In this notation, \eqref{equazfin} reads as
\begin{equation}\label{IIII}
I_{j,m}^\frac{n-1}{n} \le II^\frac{1}{n}_{j,m} + C(III_{j,m} + IV_{j,m}).
\end{equation}
We wish to show
\begin{align}
\liminf_m\liminf_j  I_{j,m} &\ge h(a)\int_{\T^n}\varphi^\frac{n}{n - 1}(x)\,dx \label{I}\\
\lim_m\lim_j  II_{j,m} &=\det\left(A(a)\right)\label{II}\\
\lim_m\lim_j  III_{j,m} &=0\label{III}\\
\lim_m\lim_j  IV_{j,m} &=0.\label{IV}
\end{align}
If we do so, then exploiting again \eqref{IIII} and letting $\varphi$ approximate the function $1$, we find
\[
h(a) \le {\det}^\frac{1}{n - 1}(A(a))
\]
as wanted. We are thus only left to show \eqref{I}-\eqref{II}-\eqref{III}-\eqref{IV} and finally our claim \eqref{imp}.
\\
\\
\indent\fbox{\textbf{Step 3:} proof of \eqref{I}}
\\
\\
We have
$$
I_{j,m}=\int_{Q_{1/m}(a)} \varphi^{\frac{n}{n-1}}\left(m(y-a)\right)\det(A_{k_j}(y))^\frac{1}{n - 1}m^n\,dy.
$$
Since  $\det^{\frac{1}{n-1}}( A_{k_j})\rightharpoonup^* \mu$, by letting $j\to \infty$ and recalling that $\mu^s\geq 0$ in \eqref{h_density}, we get 
$$
\liminf_{j \to \infty} I_{j,m}\geq  \int_{Q_{1/m}(a)} \varphi^{\frac{n}{n-1}}\left(m(y-a)\right)h(y)m^n\,dy= \int_{\T^n} \varphi^{\frac{n}{n-1}}\left(x\right)h\left(a+\frac{x}{m}\right)\,dx.
$$
Finally, since  $a\in (0,1)^n$ was a Lebesgue point for the function $h$, letting $m \to \infty$ we achieve 
$$
\liminf_{m\to \infty} \liminf_{j\to \infty}I_{j,m} \geq  h(a) \int_{\T^n}  \varphi^{\frac{n}{n-1}}(x) \,dx.
$$
\\
\\
\indent\fbox{\textbf{Step 4:} proof of \eqref{II}}
\\
\\
This is immediate since 
\[
\lim_{j\to \infty} B_{k_j,1/m} = (1-\varphi(x))A\left(a+\frac{x}{m}\right) + \varphi(x)A(a)
\] weakly in $L^1$ and 
\[
\lim_{m\to \infty}(1-\varphi(x))A\left(a+\frac{x}{m}\right) + \varphi(x)A(a) = A(a),
\]
strongly in $L^1$, since $a$ is a Lebesgue point of $A$.
\\
\\
\indent\fbox{\textbf{Step 5:} proof of \eqref{III}}
\\
\\
Recall our notation $\nu_k = |\dv A_k|$ and its weak-$*$ limit $\nu$. Using again \cite[Thm. 1.40]{EVG}, we estimate
\begin{align*}
\limsup_{j \to \infty}m^{n-1}\nu_{k_j}(Q_{\frac1m}(a)) &\le \frac{1}{m}\frac{\nu(Q_{\frac1m}(a))}{(\frac1m)^n} \le \frac{C'}{m}\frac{\nu(\overline{B_{\sqrt{2}/m}(a)})}{|B_{\sqrt{2}/m}(a)|} = \frac{C'}{m}\fint_{B_{\sqrt{2}/m}(a))}g\,dx + \frac{C'}{m}\frac{\nu^s(\overline{B_{\sqrt{2}/m}(a)})}{|B_{\sqrt{2}/m}(a)|},
\end{align*}
for some positive constant $C'$. Since we chose $a \in T''$, we get that the previous expression converges to 0 as $m \to \infty$.
\\
\\
\indent\fbox{\textbf{Step 6:} proof of \eqref{IV}}
\\
\\
We have
\begin{align*}
IV_{j,m} =\int_{\T^n}\left(A_{k_j}\left(a + \frac{x}{m}\right) -A(a)),D^2\varphi\right) \phi_{k_j,\frac1m}\,dx &=\int_{\T^n}\left(A_{k_j}\left(a + \frac{x}{m}\right) -A(a),D^2\varphi\right) \left(\phi_{k_j,\frac1m}- \phi_{\frac1m}\right)\,dx\\
&+ \int_{\T^n}\left(A_{k_j}\left(a + \frac{x}{m}\right) -A(a),D^2\varphi\right) \phi_{\frac1m}\,dx.
\end{align*}
The first term can be estimated as
\begin{align*}
&\left| \int_{\T^n} \left(A_{k_j}\left(a + \frac{x}{m}\right) -A(a),D^2\varphi\right) \left(\phi_{k_j,\frac1m}- \phi_{\frac1m}\right)\,dx\right| \\
&\qquad\qquad\qquad\qquad\quad \leq \left\|\phi_{k_j,\frac1m}- \phi_{\frac1m}\right\|_{C^0(\T^n)}\|D^2\varphi\|_{C^0(\T^n)}\int_{\T^n}\left|A_{k_j}\left(a + \frac{x}{m}\right) - A(a)\right|\,dx \\
&\qquad\qquad\qquad\qquad\quad=\left\|\phi_{k_j,\frac1m}- \phi_{\frac1m}\right\|_{C^0(\T^n)}\|D^2\varphi\|_{C^0(\T^n)}m^n\int_{Q_{\frac1m}(a)}\left|A_{k_j}(x) - A(a)\right|\,dx.
\end{align*}
Since $x\mapsto \|A_{k_j}(x) - A(a)\|$ is bounded in $L^1(Q_{\frac1m}(a))$ and by the uniform convergence of $\phi_{k_j,\frac1m}$ to $\phi_{\frac1m}$, we infer that the last term converges to $0$ as $j \to \infty$. On the other hand, by weak $L^1$ convergence,
\[
\int_{\T^n}\left(A_{k_j}\left(a + \frac{x}{m}\right) -A(a),D^2\varphi\right) \phi_{\frac1m}\,dx \to \int_{\T^n}\left(A\left(a + \frac{x}{m}\right) -A(a),D^2\varphi\right) \phi_{\frac1m}\,dx
\]
as $j \to \infty$. Now, since $\varphi$ is smooth and by \eqref{unif}, we can estimate for some $C > 0$:
\[
\left|\int_{\T^n}\left(A\left(a + \frac{x}{m}\right) -A(a),D^2\varphi\right) \phi_{\frac1m}\,dx\right| \le C\int_{\T^n}\left|A\left(a + \frac{x}{m}\right) - A(a)\right|\,dx.
\]
Since $a$ is a Lebesgue point for $A$, the last term converges to $0$ as $m\to \infty$.
\\
\\
\indent\fbox{\textbf{Step 7:} proof of \eqref{imp}}
\\
\\
By definition, we have 
\[
S_{k,R} = \lambda_{k,R}\cof\left(\int_{\T^n}B_{k,R}(x)\,dx\right).
\]
Therefore it suffices to prove separately that
\begin{equation}\label{imp1}
\limsup_{R\to 0}\limsup_{k\to \infty}\left|\cof\left(\int_{\T^n}B_{k,R}(x)\,dx\right)\right| < +\infty
\end{equation}
and
\begin{equation}\label{imp2}
\limsup_{R\to 0}\limsup_{k\to \infty}\lambda_{k,R} < +\infty.
\end{equation}

We start with $\eqref{imp1}$. As in Step 4, the weak convergence of $A_k$ to $A$ in $L^1$ and the fact that $a$ is a Lebesgue point for $A$ imply that
\[
\lim_{R\to 0}\lim_{k \to \infty}\int_{\T^n}B_{k,R}(x)\,dx = A(a).
\]
Hence
\[
\limsup_{R\to 0}\limsup_{k\to \infty}\left|\cof\left(\int_{\T^n}B_{k,R}(x)\,dx\right)\right| = \lim_{R\to 0}\lim_{k\to \infty}\left|\cof\left(\int_{\T^n}B_{k,R}(x)\,dx\right)\right| =  |\cof(A(a))| < + \infty,
\]
where the last inequality is again justified by $a \in T'$. Finally, we compute \eqref{imp2}. By definition
\[
\displaystyle\lambda_{k,R}=\frac{\left(\int_{\T^n}\det(B_{k,R})^{\frac{1}{ n - 1}}(x)\,dx\right)^{\frac{1}{n}}}{\left(\det(\int_{\T^n}B_{k,R}(x)\,dx)\right)^{\frac{n - 1}{n}}}.
\]
Analogously to the estimate of $\gamma_{k,R}$ of \eqref{bound}, we have
\[
\left(\det\left(\int_{\T^n}B_{k,R}(x)\,dx\right)\right)^{\frac{n - 1}{n}} \ge \varepsilon^{n - 1}.
\]
Therefore, to conclude the proof, we just need to show that
\[
\limsup_{R \to 0} \limsup_{k \to \infty} \int_{\T^n}\det(B_{k,R})^{\frac{1}{ n - 1}}(x)\,dx < +\infty.
\]
First note that $A(a) \le |A(a)|\Id_n$, and consequently estimate
\[
\det(B_{k,R}) \le \det(\varphi(x)A_k(a+Rx) + (1 -\varphi(x))|A(a)|\Id_n) = P_{-\varphi(x) A_k(a + Rx)}((1 - \varphi(x))|A(a)|),
\]
where $P_{-\varphi(x) A_k(a + Rx)}$ is the characteristic polynomial of $-\varphi(x) A_k(a + Rx)$. Recall the functions $M_{n-i}$ introduced in \eqref{M}. By the structure of the characteristic polynomial and the subadditivity of the function $t \mapsto t^{\frac{1}{n - 1}}$, we can bound
\begin{align*}
\det(B_{k,R})^\frac{1}{n - 1}(x) &\le |P_{-\varphi(x) A_k(a + Rx)}((1 - \varphi(x))|A(a)|)|^{\frac{1}{n - 1}} \\
&\le \sum_{i = 0}^n\left[(1 - \varphi(x))^i|A(a)|^iM_{n - i}(\varphi(x) A_k(a + Rx))\right]^{\frac{1}{n - 1}}.
\end{align*}
Since $M_{n - i}$ is $n - i$ homogeneous, $M_{n - i}(\varphi(x) A_k(a + Rx)) = \varphi^{n - i}(x)M_{n - i}(A_k(a + Rx))$. Hence
\[
\det(B_{k,R})^\frac{1}{n - 1}(x) \le \sum_{i = 0}^n\left[(1 - \varphi(x))^i|A(a)|^i\varphi^{n - i}(x)M_{n - i}(A_k(a + Rx))\right]^{\frac{1}{ n -1}}.
\]
Thus by \eqref{minors_weak_compact} and by recalling also that $M_n^{\frac{1}{n-1}}(A_k(x))=\det^{\frac{1}{n-1}}(A_k(x))\weak \mu(x)=h(x)\,dx+\mu^s(x)$, by letting $k\rightarrow \infty$ we find that 
\[
\int_{\T^n}\left[(1 - \varphi)^i\varphi^{n - i}M_{n - i}(A_k(a + Rx))\right]^{\frac{1}{ n -1}}\,dx \to \int_{\T^n}\left[(1 - \varphi)^i\varphi^{n - i}\right]^{\frac{1}{n - 1}}m_{n-i}(a+Rx)\,dx,\quad \forall i\neq 0,
\]
and
\begin{align*}
\int_{\T^n}\left[\varphi^{n }M_{n }(A_k(a + Rx))\right]^{\frac{1}{ n -1}}\,dx &= \frac{1}{R^n}\int_{\T^n}\left[\varphi^{n }\left(\frac{x-a}{R}\right)M_{n}(A_k(x))\right]^{\frac{1}{ n -1}}\,dx\\
& \to \int_{\T^n}\varphi^{\frac{n}{n - 1}}(x)h(a+Rx)\,dx+\frac{1}{R^n}\int_{\T^n} \varphi^\frac{n}{n-1}\left(\frac{x-a}{R}\right)\,d\mu^s(x)\\
& = \int_{\T^n}\varphi^{\frac{n}{n - 1}}(x)h(a+Rx)\,dx+\frac{1}{R^n}\int_{Q_R(a)} \varphi^\frac{n}{n-1}\left(\frac{x-a}{R}\right)\,d\mu^s(x).
\end{align*}
Recall that, since $a \in T'$, $a$ is a Lebesgue point for $x \mapsto m_{n-i}(x), \forall i,$ and for $x \mapsto h(x)$. Furthermore, again since $a \in T'$, it is also a zero-density point for $\mu^s$. Hence, letting $R \to 0^+$, we deduce
\begin{align*}
\limsup_{R\to 0^+}\limsup_{k \to \infty}\int_{\T^n}\det(B_{k,R})^{\frac{1}{n - 1}}(x)\,dx &\leq \sum_{i = 1}^n m_{n-i}(a)\int_{\T^n}\left[(1 - \varphi(x))^i|A(a)|^i\varphi^{n - i}(x)\right]^{\frac{1}{ n -1}}\,dx\\
&\quad+h(a)\int_{\T^n} \varphi^{\frac{n}{n-1}} (x)\,dx\\
& \le \sum_{i = 1}^n m_{n-i}(a)|A(a)|^{\frac{i}{n-1}} + h(a),
\end{align*}
the last inequality being true since $0 \le\varphi(x) \le 1, \forall x \in \T^n$. The last term is finite once again by our choice $a\in T'$. The proof of this claim is concluded, and hence we have shown Theorem \ref{introleb}.

\section{A complete upper semicontinuity result in the critical case}\label{charac}

Aim of this section is to show Theorem \ref{intro:charact_singular_part}. Combining this with Theorem \ref{introleb}, we readily obtain Corollary \ref{intro:full_USC}.

\subsection{Proof of Theorem \ref{intro:charact_singular_part}}
We will actually show a more general statement, in that we will not assume $\{A_k\}$ to be defined on $\T^n$, but on $T \doteq P\T^n$, for some $P \in \text{SL}(n)$, i.e. $P \in \R^{n\times n}$ and $\det(P) = 1$. In other words, if $\T^n$ is the quotient $\R^n/\mathbb{Z}^n$, then $T$ is the quotient $\R^n/(P\mathbb{Z}^n)$. In particular, we will be interested in showing that the constant appearing in the last inequality is independent of $P$. We will exploit this in Section \ref{addit} in a crucial way. The proof is divided into steps.
\\
\\
\fbox{\textbf{Step 1:} a useful inequality.}
\\
\\
Let $A' \in L^1(T)$ with $\dv(A') \in \mathcal{M}(T;\R^n)$. Fix $x \in T$ and $r \in (0,1)$. Consider $A'_\eps = A'\star \rho_\eps$, the mollification of $A'$ at scale $\eps>0$. Denote $f_\eps(x) \doteq {\det}^\frac{1}{n-1}(A'_\eps)(x).$ We can directly employ  \cite[Thm. 2.3]{SER} for $A'_\eps$ on $\Omega = B_r(x)$ to find
\[
\int_{B_r(x)}f_\eps \,dy \le C\left(\int_{\partial B_r(x)} |A'_\eps n_{\partial B_r}|\,d\sigma + \int_{B_r(x)}|\dv(A'_\eps)|\,dy\right)^\frac{n}{n -1}.
\]
Here $C> 0$ is a dimensional constant that may vary from line to line and $n_{\partial B_r}$ is the unit normal to $\partial B_r(x)$. Rewrite this as
\[
\left(\int_{B_r(x)}f_\eps(x)\,dy\right)^{\frac{n - 1}{n}} \le C\left( \int_{\partial B_r(x)} |A'_\eps|\,d\sigma + \int_{B_r(x)}|\dv(A'_\eps)|\,dy\right).
\]
Integrating between $0$ and $2R$ with $2R \le 1$, we obtain
\begin{equation}\label{still}
R\left(\int_{B_R(x)}f_\eps(x)\,dy\right)^{\frac{n - 1}{n}} \le C\left(\int_{B_{2R}(x)} |A'_\eps|\,dy + R\int_{B_{2R}(x)}|\dv(A'_\eps)|\,dy\right)
\end{equation}
Let $E$ be the set of $R \in (0,1/2)$ such that $|\dv(A')|(\partial B_{2R}(x)) = 0$. Since by assumption $|\dv(A')|$ is a finite measure, $E^c$ is at most countable, and in particular $E$ is dense in $(0,1/2)$. By \cite[Thm.  4.36]{MAG} we have $|\dv A'_{\eps}| \weak |\dv A'|$. We can now exploit \cite[Thm. 1.40]{EVG} to compute, for each $R \in E$, the limit of \eqref{still} as $\eps \to 0$:
\begin{equation}\label{useful}
\int_{B_R(x)}{\det}^\frac{1}{n-1}(A'(x))\,dy \le C\left(\frac{1}{R}\int_{B_{2R}(x)}|A'|\,dy\right)^\frac{n}{n - 1} + C|\dv(A')|^\frac{n}{n-1}(B_{2R}(x)).
\end{equation}
We can write the latter for every element of the sequence $\{A_k\}_{k \in \N}$ getting
\begin{equation}\label{usefulk}
\int_{B_R(x)}{\det}^\frac{1}{n-1}(A_k(x))\,dy \le C\left(\frac{1}{R}\int_{B_{2R}(x)}|A_k|\,dy\right)^\frac{n}{n - 1} + C|\dv(A_k)|^\frac{n}{n-1}(B_{2R}(x)),
\end{equation}
Estimate \eqref{usefulk} is valid for all $R$ such that, for all $k$,
\begin{equation}\label{Ak2R}
|\dv A_k|(\partial B_{2R}(x)) = 0.
\end{equation}
We let $I \subset (0,1/2)$ be the set of all those $R \in (0,1/2)$ such that \eqref{Ak2R} is valid for all $k \in \N$ and for which
\begin{equation}\label{R2R}
\mu(\partial{B_R(x)}) = 0,\; \mu(\partial{B_{2R}(x)}) = 0,\; \nu(\partial B_R(x)) = 0,\; \nu(\partial B_{2R}(x)) = 0.
\end{equation}
Of course, $I$ depends on $x$. Furthermore, since all the measures involved are finite, we notice that $I^c$ is at most countable, and hence $I$ is dense in $(0,1/2)$. By assumption \eqref{morereglim}, $|A_k|$ weakly converges in $L^1$ to a function $u \in L^\frac{n}{n-1}(T)$. Thus for all $R \in I$ we obtain
\begin{equation}\label{initial}
\mu(\overline{B_{R}(x)}) \le C\left(\frac{1}{R}\int_{B_{2R}(x)}u \,dy\right)^\frac{n}{n - 1} + C\nu(\overline{B_{2R}(x)})^\frac{n}{n-1}.
\end{equation}
By Jensen inequality, we estimate
\[
\left(\frac{1}{R}\int_{B_{2R}(x)}u \,dy\right)^\frac{n}{n - 1} \le C\int_{B_{2R}(x)}u^\frac{n}{n-1}\,dy.
\]
Here and in what follows, we will adopt the following convention: $C_T$ is a constant that depends on $T$, while $C$ does not. Both constants may vary from line to line. Decomposing
$\nu = g\,dy + \nu^s$, we can estimate
\begin{align*}
\nu(\overline{B_{2R}(x)})^\frac{n}{n-1} &\le C\left(\int_{B_{2R}(x)}g\,dy\right)^\frac{n}{n-1}+ C\nu^s(\overline{B_{2R}(x)})^\frac{n}{n-1}\le C_T\int_{B_{2R}(x)}g\,dy + C\nu^s(\overline{B_{2R}(x)})^\frac{n}{n-1}.
\end{align*}
Combining these inequalities, we obtain
\[
\mu^s(\overline{B_R(x)}) \le \mu(\overline{B_R(x)}) \le C_T\int_{B_{2R}(x)}\left(u^\frac{n}{n-1} +g\right)  \,dy + C\nu^s(\overline{B_{2R}(x)})^\frac{n}{n-1}.
\]
Calling $v \doteq C_T(u^\frac{n}{n-1} +g) \in L^1(T)$, we obtain for all $R \in I$ our fundamental inequality
\begin{equation}\label{MI}
\mu^s(\overline{B_R(x)}) \le  \int_{B_{2R}(x)}v(y) \,dy+ C\nu^s(\overline{B_{2R}(x)})^\frac{n}{n-1}.
\end{equation}
Observe that we also have the weaker inequality for all $R \in I$:
\begin{equation}\label{MIweak}
\mu^s(\overline{B_R(x)}) \le  \int_{B_{2R}(x)}v(y) \,dy+ C\nu^s(\overline{B_{2R}(x)})^\frac{1}{n-1}\nu^s(\overline{B_{2R}(x)} )\le \int_{B_{2R}(x)}v(y) \,dy+ C_T\nu^s(\overline{B_{2R}(x)}).
\end{equation}
Without loss of generality, we will assume $v > 0$ a.e. in $\T^n$.
\\
\\
\fbox{\textbf{Step 2:} $\mu^s$ is absolutely continuous with respect to $\nu^s$.}\\

We set $\theta$ to be the auxiliary measure defined by $\theta = v\,dy + \nu^s.$ We can employ Theorem \ref{RNdec} to write $\mu^s = f\,d\theta + \beta$. We first show that $\beta \equiv 0$. Suppose this is not the case. First of all, \eqref{MIweak} immediately implies that if $x \notin \spt(\theta)$, then $x \notin \spt(\beta)$. Thus, from Theorem \ref{RNdec} we find that $\beta$ is concentrated on the set $E$ where
\begin{equation}\label{inf1}
\liminf_{R \to 0^+}\frac{\beta(\overline{B_R(x)})}{\theta(\overline{B_R(x)})} = + \infty.
\end{equation}
By Proposition \ref{asdoub}, we may pick a point $x_0\in E \subset \spt(\beta)$ for which
\[
\limsup_{R \to 0^+}\frac{\beta(B_R(x_0))}{\beta(B_{2R}(x_0))} \ge 2^{-n}.
\]
In particular, we can find a sequence $\{R_j\}_{j \in \N}$ with $R_j,2R_j \in I$ for all $j \in \N$ such that
\begin{equation}\label{noninf1}
\liminf_j \frac{\beta(B_{R_j}(x_0))}{\beta(B_{2R_j}(x_0))} = \liminf_j \frac{\beta(\overline{B_{R_j}(x_0)})}{\beta(\overline{B_{2R_j}(x_0)})} \ge 2^{-n}.
\end{equation}
For such a point $x_0 \in E$, we may use \eqref{MIweak} to bound
\[
\frac{\beta(\overline{B_{R_j}(x_0)})}{\beta(\overline{B_{2R_j}(x_0)})}\frac{\beta(\overline{B_{2R_j}(x_0)})}{\theta(\overline{B_{2R_j}(x_0)})} \leq \frac{\mu^s(\overline{B_{R_j}(x_0)})}{\beta(\overline{B_{2R_j}(x_0)})}\frac{\beta(\overline{B_{2R_j}(x_0)})}{\theta(\overline{B_{2R_j}(x_0)})}=\frac{\mu^s(\overline{B_{R_j}(x_0)})}{\theta(\overline{B_{2R_j}(x_0)})}  \le C_T.
\]
Combining \eqref{inf1} and \eqref{noninf1} we obtain a contradiction. This shows $\beta \equiv 0$, i.e. that $\mu^s = f\,d\theta$.
Since $\theta = v\,dy + \nu^s$, we then have $\mu^s = fv dy + f\nu^s$. As $\mu^s$ and $\nu^s$ are singular with respect to $\mathcal{L}^n$, we find that $fv = 0$ for $\mathcal{L}^n$ a.e. point. It follows that
$\mu^s = f\,d\nu^s$.
\\
\\
\fbox{\textbf{Step 3: }$\mu^s$ is supported on $\{x_i\}_{i \in \N}$ where $\nu^s(\{x_i\})>0$.}\\

We know that $\mu^s = f\,d\nu^s$. Define
\[
V \doteq \big\{x \in T: \nu^s(\{x\}) = 0\big\}.
\]
Our aim is to show that $f(x) = 0$ for $\nu^s$-a.e. $x \in V$. We consider the set $A_{\nu^s}$ of Proposition \ref{asdoub}. Set $\alpha \doteq vdy$. Let $F$ be the set of all $x$ such that
\[
\limsup_{R \to 0^+}\frac{\alpha(B_R(x))}{\nu^s(\overline{B_R(x)})} = 0.
\]
Since $v>0$ a.e. in $T$, then obviously $\spt (\alpha)=T$. Then by Theorem \ref{RNdec} and since $\nu^s \perp \alpha$, we know that $\nu^s$ has to be concentrated on the set 
$$
\left\{x \in T \,:\, \liminf_{r \to 0^+}\frac{\nu^s(\overline{B_r(x)})}{\alpha(B_r(x))} = + \infty\right\},
$$
which in turn proves that the set $F$ is of full $\nu^s-$measure. Finally, let $D$ to be the set of the points $x \in T$ such that
\[
\frac{1}{\nu^s(\overline{B_R(x)})}\int_{\overline{B_R(x)}}f(y)\,d\nu^s(y) \to f(x).
\]
By \cite[Thm. 1.32]{EVG}, this is once again a set of full $\nu^s$ measure in $T$. To conclude this step, take any $x \in A_{\nu^s}\cap F\cap D \cap V$. Let $R_j \in I$ be a sequence such that
\begin{equation}\label{liminf_lowerbound}
\liminf_{j\rightarrow \infty} \frac{{\nu^s}(\overline{B_{R_j}(x)})}{{\nu^s}(\overline{B_{2R_j}(x)})} \ge 2^{-n}.
\end{equation}
 For any such $R_j$, use \eqref{MI} to find
$$
\int_{\overline{B_{R_j}(x)}} f\,d\nu^s=\mu^s(\overline{B_{R_j}(x)})\leq C_T \alpha(\overline{B_{2R_j}(x)})+C\nu^s(\overline{B_{2R_j}(x)})^\frac{n}{n-1},
$$
and consequently, dividing by $\nu^s(\overline{B_{R_j}(x)})$ and using \eqref{liminf_lowerbound}, we obtain that for all $j$ large enough
\begin{align}
\frac{1}{\nu^s(B_{R_j}(x))}\int_{\overline{B_{R_j}(x)}} f\,d\nu^s&\leq C_T\frac{\alpha(\overline{B_{2R_j}(x)})}{\nu^s(\overline{B_{2R_j}(x)})} \frac{\nu^s(\overline{B_{2R_j}(x)})}{\nu^s(\overline{B_{R_j}(x)})}+ C \frac{\nu^s(\overline{B_{2R_j}(x)})}{\nu^s(\overline{B_{R_j}(x)})}   \nu^s(\overline{B_{2R_j}(x)})^\frac{1}{n-1} \nonumber \\
&\leq 2^{n+1}  \left(C_T\frac{\alpha(\overline{B_{2R_j}(x)})}{\nu^s(\overline{B_{2R_j}(x)})} + C  \nu^s(\overline{B_{2R_j}(x)})^\frac{1}{n-1}  \right)\label{almost_done}, 
\end{align}
from which, by letting $j\rightarrow \infty$ and since $x \in  F\cap D \cap V$,  we conclude
$$
f(x)=0 \quad \text{on } A_{\nu^s}\cap F\cap D \cap V.
$$
Since the set $A_{\nu^s}\cap F\cap D$ is of $\nu^s$ full measure, denoting by $\{x_i\}_{i\in \N}$ the set of points such that $\nu^s(\{x_i\})>0$ we get that $f$, and thus $\mu^s$, is concentrated on $\{x_i\}_{i\in \N}$, or in other words
$$
\mu^s=\sum_{i=1}^\infty \mu_i \delta_{x_i},
$$
for some coefficients $\mu_i\geq 0$. By choosing $x = x_i$ in \eqref{MI} and sending $R \to 0^+$, we immediately find
$\mu_i\leq C\nu^s(\{x_i\})^\frac{n}{n-1}$, 
for some purely dimensional constant $C=C(n)>0$.

\qed

\subsection{Improved conclusion of Theorem \ref{intro:charact_singular_part}}\label{addit}
To show how to strengthen Theorem \ref{intro:charact_singular_part}, we first need a technical tool, which is essentially taken from \cite[Prop. 4.1]{AB}. Consider a sequence of vector-valued measures on $\T^n$, $V_k \doteq X_kd\nu_k$, where $X_k \in L^1(\T^n,\R^n; \nu_k)$ and $|X_k(x)| \neq 0$ for $\nu_k$-a.e. $x \in \T^n$. Assume that 
\[
\sup_k\|V_k\|_{\mathcal{M}(\T^n;\R^n)}= \sup_k\int_{\T^n}|X_k|\,d\nu_k < + \infty.
\]
For $\varphi \in C^0(\T^n\times\mathbb{S}^{n-1})$, we set
\begin{equation}\label{complete}
L_k(\varphi) = \int_{\T^n}\varphi\left(x,\frac{X_k(x)}{|X_k(x)|}\right)|X_k(x)|d\nu_k(x).
\end{equation}
As $|L_k(\varphi)| \le C\|\varphi\|_{C^0}$, for every $ k \in \N$, the Riesz Representation Theorem \cite[Thm. 1.38]{EVG}, provides us with a finite measure $\alpha_k$ on $\T^n\times \mathbb{S}^{n-1}$ with $\|\alpha_k\|_{\mathcal{M}(\T^n\times \mathbb{S}^{n-1})} \le C$ such that 
\[
L_k(\varphi) = \int_{\T^n\times \mathbb{S}^{n-1}}\varphi(x,z)d\alpha_k(x,z).
\]
By the weak-$*$ compactness of Radon Measures, see \cite[Thm. 1.41]{EVG}, we can find a subsequence $\{\alpha_{k_j}\}$ and a Radon measure $\alpha$ on $\T^n \times \mathbb{S}^{n-1}$ such that
$\alpha_{k_j}\weak \alpha$.
Finally, we consider the disintegration of $\alpha$ 
\[
\alpha(\varphi) = \int_{\T^n}\left(\int_{\mathbb{S}^{n-1}}\varphi(x,z)d\nu_x(z)\right)d\theta(x),
\]
for a family of probability measures $\nu_x$ on $\mathbb{S}^{n-1}$ and a Radon measure $\theta$ on $\mathbb T^n$. We write $\alpha = (\{\nu_x\}_{x \in \T^n},\theta)$. Observe that, if we have that $\nu_k \weak \nu$, i.e. if the variations $|V_k|$ converge to $\nu$, then $\theta = \nu$, as can be seen easily by noticing that, if $\varphi(x,z) = h(x)$,
\[
L_{k_j}(\varphi) = \int_{\T^n} h(x)|X_{k_j}(x)|d\nu_{k_j}(x) = |V_{k_j}|(h).
\]
We then give the following
\begin{Def}\label{cc}
Let $V_k = X_kd\nu_k$ be a sequence of vector-valued measures defined on $\T^n$ with values in $\R^n$ and with $|X_k| \neq 0$ for $\nu_k$-a.e. $x \in \T^n$, $\sup_k \int_{\T^n} |X_k|\,d\nu_k < + \infty$. Then, we say that $V_k$ converges completely to the couple $(\{\nu_x\}_{x \in T},\nu)$ as above if
\begin{enumerate}
\item $\nu_k \weak \nu$;
\item the functionals $L_k$ defined in \eqref{complete} are such that 
\[
L_k(\varphi) \to \int_{\T^n}\left(\int_{\mathbb{S}^{n-1}}\varphi(x,z)d\nu_x(z)\right)d\nu(x),\quad \forall \varphi \in C^0(\T^n\times \mathbb S^{n-1}).
\]
\end{enumerate}
\end{Def}
Let us now go back to our problem. Consider a sequence $\{A_k\}_k$ fulfilling the assumptions of Theorem \ref{intro:charact_singular_part}. We represent $\dv A_k(x) = X_kd\nu_k$ as above and consider $\nu$ to be the weak-$*$ limit of $\{\nu_k\}_k$. Up to a (non-relabeled) subsequence, we assume that $\{\dv A_k\}_k$ converges completely to $(\{\nu_x\}_{x \in \T^n},\nu)$. Now we consider any matrix $P \in \text{SL}(n) = \{X \in \R^{n\times n}: \det(X) = 1\}$ and set, as in \cite[Lemma 1.1]{SER},
\[
B_k(y) \doteq PA_k(P^{-1}y)P^T, \quad \forall y \in T \doteq P\T^n.
\]
In particular, $B_k \in \Sym(n)^+$ for a.e. $y \in T$ and $\dv B_k$ is a vector-valued measure represented by
\[
\dv B_k = Y_kd\beta_k = P\dv A_k(P^{-1}y) =PX_k(P^{-1}\cdot) (f_P)_{\#}(\nu_k),
\]
where $(f_P)_{\#}(\nu_k)$ is the measure defined as the pushforward through $P$ of $\nu_k$, i.e.
\[
(f_P)_{\#}(\nu_k)(h) = \int_{\T^n}h(Px)d\nu_k(x). 
\]
Using the definition of complete convergence, we can write, for any $h \in C^0(T)$
\begin{align*}
\beta_k(h) = |\dv B_k| (h) &= \int_{T}h(y)\left|PX_k(P^{-1}y)\right| \,d(f_P)_{\#}(\nu_k) \\
&= \int_Th(Px)|PX_k(x)|d\nu_k \to \int_{T}h(Px) \left(\int_{\mathbb{S}^{n-1}}|Pz|d\nu_x\right)d\nu.
\end{align*}
Hence
\begin{equation}\label{carbeta}
\beta_k \weak \beta = g(\cdot)(f_P)_{\#}(\nu), \quad \text{where }g(y) = \int_{\mathbb{S}^{n-1}}|Pz|d\nu_{P^{-1}y}, \text{ for $(f_P)_{\#}(\nu)$-a.e. } y \in T. 
\end{equation}
Furthermore, if $\mu_k = \det^{\frac{1}{n - 1}}(A_k)$ converges weakly-$*$ to $\mu$, then $\mu_k^B \doteq  \det^{\frac{1}{n - 1}}(B_k) =  (f_P)_{\#}(\mu_k)$ converges weakly-$*$ to $\mu^B = (f_P)_{\#}(\mu)$. Now we employ Theorem \ref{intro:charact_singular_part} on $\{B_k\}_k$. Notice that in our previous section we showed the result for matrix fields defined on $P\T^n$, with a constant $C$ independent of $P$. Hence we find
\[
(\mu^B)^s \le C\sum_{i}\beta(\{y_i\})^\frac{n}{n - 1}\delta_{y_i}.
\]
Moreover, we observe that $\nu^s$ contains $\delta_{x_i}$ with weight $\nu^s(\{x_i\})$ if and only if $\beta^s$ contains $\delta_{Px_i}$ with weight $\left(\int_{\mathbb{S}^{n-1}}|Pz|d\nu_{x_i}\right)\nu^s(\{x_i\})$. Combining the latter with the fact that $\mu^B = (f_P)_{\#}(\mu)$, we find the estimate
\[
\mu^s \le C\sum_{i}\left(\int_{\mathbb{S}^{n-1}}|Pz|d\nu_{x_i}\right)^\frac{n}{n - 1}\nu^s(\{x_i\})^\frac{n}{n - 1}\delta_{x_i}.
\]
If $P = \Id$, we get back to the estimate of Theorem \ref{intro:charact_singular_part}. We infer two corollaries from this proof.

\begin{Cor}\label{cor1}
Let $\{A_k\}_k$ fulfill the assumptions of Theorem \ref{intro:charact_singular_part}. Suppose in addition that $\{\dv A_k\}$ converge completely to $(\{\nu_x\}_{x \in \T^n},\nu)$ in the sense of Definition \ref{cc}. Then, there exists a dimensional constant $C= C(n) > 0$ such that the following holds. If $\{x_i\}_{i\in \N}$ is the countable set  of points in $\T^n$ such that $\nu^s(\{x_i\}) > 0$, then
\[
\mu^s\leq C(n) \inf_{P \in SL(n)}\sum_{i=1}^\infty \left(\int_{\mathbb{S}^{n-1}}|Pz|d\nu_{x_i}\right)^\frac{n}{n - 1}\nu^s(\{x_i\})^\frac{n}{n-1}\delta_{x_i}\quad \text{as measures}.
\]
\end{Cor}

Notice that the additional requirement that $\{\dv A_k\}_k$ converges completely is not adding anything to the assumptions of Theorem \ref{intro:charact_singular_part}, since we can always achieve it after taking a subsequence. Let us move to the second corollary. First, for all $v \in \mathbb{S}^{n-1}$, we consider the measure $(\dv A_k,v)$ and its variation measure $|(\dv A_k,v)|$. Theorem \ref{intro:charact_singular_part} tells us that, in many cases, concentration in the limit of $\D(A_k)$ is only due to Dirac Delta's contained in the limit of $|\dv A_k|$. The reasoning above actually yields that, if a Dirac Delta is contained in $\mu^s$, then it must be contained in any weak-$*$ limit of $|(\dv A_k,v)|$ for any direction $v \in \mathbb S^{n-1}$. Let us show why.

\begin{Cor}\label{cor2}
Let $\{A_k\}_k$ fulfill the assumptions of Theorem \ref{intro:charact_singular_part}. Suppose that for some $v \in \mathbb{S}^{n-1}$ and for some subsequence
\[
|(\dv A_{k_j},v)| \weak \alpha
\]
and $\alpha$ is diffuse, i.e. $\alpha(\{x\}) = 0$ for all $x \in \T^n$. Then, $\mu^s \equiv 0$.
\end{Cor}
\begin{proof}
We will not relabel subsequences. Extract a further subsequence and assume that $\{\dv A_k\}_k$ converges completely in the sense of Definition \ref{cc} to $(\{\nu_x\}_{x \in \T^n},\nu)$. Notice that if $\dv A_k = X_kd\nu_k$, then
\[
(\dv A_k, v) = (X_k,v)d\nu_k
\]
and hence $|(\dv A_k, v)| = |(X_k,v)|d\nu_k.$ As $\{ \dv A_k\}_{k}$ converges completely to $(\{\nu_x\}_{x \in \T^n},\nu)$, we see that
\[
|(\dv A_k, v)| \weak  f\nu, \quad \text{where } f(x) = \int_{\mathbb{S}^{n-1}}|(z,v)|d\nu_x.
\]
By uniqueness of the limit, $f \nu = \alpha$. The assumption that $\alpha$ is diffuse tells us that 
\[
\nu(\{x_i\}) > 0 \Rightarrow \int_{\mathbb{S}^{n-1}}|(z,v)|d\nu_{x_i}(z) = 0.
\]
Therefore, $\spt(\nu_{x_i}) \subset v^\perp$, for all $\text{$x_i$ s.t. } \nu(\{x_i\}) > 0.$ Complete $v$ to an orthonormal basis of $\R^n$,  say $\{v_1,\dots, v_n\}$,  with $v_1 = v$. For any $a>0$, set $P_a \in \text{SL}(n)$ to be
\[
P_a = a^{-(n-1)} v_1\otimes v_1 + a\sum_{i = 2}^n v_i\otimes v_i = a^{-(n-1)}v\otimes v +  a\pi_{v^\perp},
\]
where $\pi_{v^\perp}$ is the orthogonal projection on $v^\perp$. Let $x_i$ be such that $\nu(\{x_i\}) > 0$. Then
\begin{equation}\label{estimatz}
\int_{\mathbb{S}^{n-1}}|P_az|\,d\nu_{x_i} \le a^{-(n-1)}\int_{\mathbb{S}^{n-1}}|(z,v)|\,d\nu_{x_i} + a\int_{\mathbb{S}^{n-1}}|\pi_{v^\perp}z|\,d\nu_{x_i} = a\int_{\mathbb{S}^{n-1}}|\pi_{v^\perp}z|\,d\nu_{x_i} \le a,
\end{equation}
since $\nu_{x_i}$ is a probability measure for all $i$. Now use Corollary \ref{cor2} to find, for all $a > 0$, in the sense of measures,
\[
\mu^s\leq C(n) \sum_{i=1}^\infty \left(\int_{\mathbb{S}^{n-1}}|P_az|d\nu_{x_i}\right)^\frac{n}{n - 1}\nu^s(\{x_i\})^\frac{n}{n-1}\delta_{x_i} \overset{\eqref{estimatz}}{\le}  a^\frac{n}{n - 1}C(n) \sum_{i=1}^\infty \nu^s(\{x_i\})^\frac{n}{n-1}\delta_{x_i}.
\]
This finally shows $\mu^s \equiv 0$ by letting $a\rightarrow 0$.
\end{proof}

We now  prove the last of the corollaries, i.e. Corollary \ref{cor3}. As said, this method was introduced by Lions in \cite{PLL1,PLL2,PLL3,PLL4}. A beautiful application of the method is provided by its use in the study of possible lack of strong compactness in $L^{p^*}(\T^n)$ for equibounded sequences in $W^{1,p}(\T^n)$. We refer the reader to \cite[Thm. 1.4.2]{EVAWEAK} or \cite{PLL3} for an illustration of this phenomenon. Here, $p^* = \frac{pn}{n  - p}$ for $1\le p < n$ is the Sobolev exponent. First, let us remark that, given an equibounded sequence $\{u_k\}\subset W^{1,p}(\T^n)$, we may consider the matrix field $A_k \in X_{\frac{n}{n - 1}}$ given by
\begin{equation}\label{AK}
A_k = |u_k|^{\frac{n - 1}{n - p}p}\Id.
\end{equation}
Set $\eta = \frac{n - 1}{n - p}p$. Then we have
$\dv A_k \in L^1(\T^n;\R^n)$ with
$$
\dv A_k = \eta |u_k|^{\eta - 2}u_kDu_k \quad \text{ and } \quad
 |\dv A_k| = \eta |u_k|^{\eta - 1}|Du_k| dx.
$$
Employing Corollary \ref{intro:full_USC}, it is not hard to obtain again the known concentration-compactness result for measures generated by the sequence $\{|u_k|^{p^*}\}_{k \in \N}$. We omit the details. However, a direct consequence of Corollary \ref{cor2} is Corollary \ref{cor3}, which we believe is new. Roughly speaking,  Corollary \ref{cor3} tells us that if the $L^p$ norm of any directional derivative of the sequence does not create Dirac Deltas in the limit, then $\{|u_k|^{p^*}\}_{k \in \N}$ does not concentrate, and thus it converges strongly.
\begin{proof}[Proof of Corollary \ref{cor3}]
As usual, we will not relabel subsequences. Consider $A_k =  |u_k|^{\eta}\Id_n$ as in \eqref{AK}. Then,
\[
|(\dv A_k,v)| = \eta |u_k|^{\eta - 1}|(Du_k,v)|.
\]
Passing to further subsequences, we can assume that $|(\dv A_k,v)| \weak \alpha$. Thus, for any $x \in \T^n$ and $r\in (0,1)$, by H\"older inequality for $\frac{1}{p}+\frac{1}{p'}=1$, we get 
\[
\int_{B_r(x)}|(\dv A_k,v)|dy = \eta \int_{B_r(x)} |u_k|^{\eta - 1}|(Du_k,v)|dy \le \eta\left(\int_{B_r(x)} |u_k|^{p^*}dy\right)^\frac{1}{p'}\left(\int_{B_r(x)} |(Du_k,v)|^{p}dy\right)^\frac{1}{p}.
\]
Thus, for almost all $r \in (0,1)$,
\[
\alpha(B_r(x)) \le C \gamma^{\frac{1}{p}}(B_r(x)).
\]
Now assumption \eqref{dirder} and Corollary \ref{cor2} show us that $\mu^s \equiv 0$, since $\det^{\frac{1}{n-1}} A_k=|u_k|^{p^*} \weak gdx + \mu^s $ by our assumptions. Thus, 
$$
\mu \doteq \text{w*-}\lim_{k\rightarrow \infty}\text{det}^{\frac{1}{n-1}} A_k= g\,dx.
$$
Then, lower semicontinuity of $L^q$ norms shows that $|u|^{p^*} \le g$ a.e. in $\T^n$, while the upper semicontinuity of Theorem \ref{introleb} shows $g \le |u|^{p^*}$. Thus, $g = |u|^{p^*}$ and now Br\'{e}zis-Lieb Lemma \cite{BRL} concludes the proof.
\end{proof}

\section{Conditional Hardy Regularity of $\det^\frac{1}{n-1}(A)$}\label{hardy}

For every function (or more generally a measure) $h$ we will denote by $Mh$ its maximal function, i.e.
$$
Mh(x) = \sup_{0 <R <1}\frac{1}{R^n}\int_{B_R(x)}|h|(y)\,dy.
$$
In this section we will give the proofs of Theorem \ref{intro:det_hardy}, Proposition \ref{counter} concerning the sharpness of Theorem \ref{intro:det_hardy}, or equivalently the failure of any possible conditional upper semicontinuity of $\D$ in the case $p<\frac{n}{n-1}$ with additional hypothesis on $\dv A_k$, and finally the improved strong convergence of Corollary \ref{intro:hardystrong}.

\subsection{Proof of Theorem \ref{intro:det_hardy}}

We can assume that $A$ is compactly supported in a ball $B \subset \R^n$. The estimate on $\T^n$ follows by applying this result on $\varphi A$, where $\varphi$ is a non-negative smooth cut-off of a ball $B$ with $[0,1]^n \subset B$. Set $f(x) \doteq {\det}^{\frac{1}{n-1}}(A)(x)$. According to \cite[Lemma 3]{MULDETPOS}, which is based on \cite{STE}, if $h$ is supported in $B$, then
\begin{equation}\label{iff}
\int_{B}h(x)\log(1 + h(x))\, dx < + \infty\quad  \Leftrightarrow \quad Mh \in L^1(B),
\end{equation}
with the estimate
\begin{equation*}
\int_{B}h(x)\log(1 + h(x)) \,dx \le c(B,\|Mh\|_{L^1(B)}).
\end{equation*}
Thus, it only suffices to show that
$Mf \in L^1(B)$.
We can start from \eqref{useful} for $A' = A$, and divide both sides by $R^n$ to find
\begin{equation}\label{usefulresc}
\fint_{B_R(x)}f\,dy \le C\left(\fint_{B_{2R}(x)}|A|\,dy\right)^\frac{n}{n -1} + C\left(\frac{|\dv(A)|(B_{2R}(x))}{(2R)^{n-1}}\right)^\frac{n}{n - 1},
\end{equation}
which implies, recalling the definition of $\tilde M$ from \eqref{gencond},
\[
\fint_{B_R(x)}f\,dy \le C M(|A|)^\frac{n}{n - 1}(x) + C \tilde M (|\dv A|)^\frac{n}{n - 1}(x), \quad \forall R \in E,
\]
where, we recall, $E \subset (0,1)$ is the set of $R$ such that
$|\dv A|(\partial B_{2R}(x)) = 0$.
Since $E$ is dense, we conclude
\[
Mf(x) \le C M(|A|)^\frac{n}{n - 1}(x) + C \tilde M (|\dv A|)^\frac{n}{n - 1}(x).
\]
Thus, by integrating over all $B$ we find
\[
\|Mf\|_{L^1(B)} \le C\left(\|M(|A|)\|^\frac{n}{n-1}_{L^\frac{n}{n - 1}(B)} + \| \tilde M (| \dv A|) \|^\frac{n}{n-1}_{L^\frac{n}{n - 1}(B)}\right).
\]
It is a well know result that $\forall p>1$ the maximal operator $M:L^p\rightarrow L^p$ is bounded, i.e. $\| M h\|_{L^p}\leq C \|h\|_{L^p}$,
and hence we get the desired estimate in the case of compactly supported matrix-fields. This concludes the proof.

\qed

\begin{rem}\label{vs}
Let us compare the result of Corollary \ref{intro:full_USC} with the one of Theorem \ref{intro:det_hardy}. Given a sequence $\{A_k\}_k \subset X_\frac{n}{n - 1}$ such that $A_k \rightharpoonup A$ whose divergence satisfies \eqref{gencond} uniformly in $k \in \N$, i.e.
\begin{equation}\label{MAk}
\sup_k \left\|\tilde M (|\dv A_k|)\right\|_{L^\frac{n}{n - 1}(\T^n)} \le C,
\end{equation}
we know by Theorem \ref{intro:det_hardy} that \eqref{llogl} holds uniformly in $k$. In particular, the sequence $\{{\det}^\frac{1}{n-1}(A_k)\}_k$ is bounded in $\mathcal{H}^1(\T^n)$ and thus it is equi-integrable. Now Theorem \ref{introleb} allows us to conclude weak upper semicontinuity of $\D(\cdot)$ along $\{A_k\}_k$. On the other hand, by Corollary \ref{intro:full_USC} it is sufficient to require that $\nu^s$ is diffuse in order for the weak upper semicontinuity to hold. Recall that $\nu^s$ is the singular part of the weak-$*$ limit of $\{|\dv A_k|\}_{k}$. The requirement that $\nu^s$ is diffuse seems to be weaker than \eqref{MAk} (see for instance the sufficient condition given in Proposition \ref{p:tildeMdivA}), and hence the two results of Corollary \ref{intro:full_USC} and Theorem \ref{intro:det_hardy} are expressing related but different properties of weakly convergent sequences in $X_{\frac{n}{n - 1}}$: on the one hand Corollary \ref{intro:full_USC} implies the upper semicontinuity if $\nu^s$ has no atoms, while the stronger assumptions in Theorem \ref{intro:det_hardy} yield the stronger conclusion that the  sequence $\{{\det}^\frac{1}{n-1}(A_k)\}_k$ is bounded in $\mathcal{H}^1(\T^n)$, which in particular implies the upper semicontinuity of $\D$.
 \end{rem}

\subsection{Failure of USC and Hardy regularity if $p<\frac{n}{n-1}$}

Here we prove that in the subcritical case $p<\frac{n}{n-1}$ no additional hypothesis on the divergence of $\{A_k\}_k$ can provide the weak upper semicontinuity of the functional $\D$, and hence if $p < \frac{n}{n - 1}$ the pointwise estimate \eqref{hest} is indeed optimal, independently on how regular $\{\dv A_k\}_k$ is. We will then comment on how this also yields that for $p < \frac{n}{n - 1}$ estimate \eqref{llogl} is in general false, see Remark \ref{exp}.
\begin{prop}\label{counter}
Let $p<\frac{n}{n-1}$. There exists a sequence $\{A_k\}_{k\in \N}\subset X_p$ and $A\in X_p$ with the following properties:
\begin{itemize}
\item[(i)]  $\dv A_k =0$ for every $k\in\N$;
\item[(ii)] $A_k\rightarrow A$ in $L^p(\T^n)$;
\item[(iii)] $\limsup_{k\rightarrow \infty} \mathbb{D} (A_k)=\lim_{k\rightarrow \infty} \mathbb{D} (A_k)> \mathbb{D} (A)$.
\end{itemize}
\end{prop}

\begin{proof}
We will construct the sequence $\{A_k\}_{k\in \N}$ by adapting the functions $f_\alpha$ constructed in \cite[Lemma 3]{DRT} with $\alpha=\frac{1}{k}\rightarrow 0$. Identify $\T^n$ with the cube $[-2,2]^n\subset \R^n$. Define the convex function $f_k:\R^n\rightarrow \R$ by

\[
f_k(x)\doteq
\begin{cases}
|x|^{1 + \frac{1}{k}} + \frac{1-k}{2k}, &\text{ if } |x| \le 1,\\
\frac{1+k}{2k}|x|^2, &\text{ if } |x| > 1.
\end{cases}
\]
From \cite[Lemma 4]{DRT} we have that $f_k\in W^{2,p}_{\loc}(\R^n)$ for all $k\in \N$ if $p<n$ and moreover its distributional Hessian is given by 
$$
D^2f_k(x)\doteq
\begin{cases}
\frac{1+k}{k}\left(|x|^{\frac{1}{k} - 1}\Id_n + (\frac{1}{k} - 1)|x|^{\frac{1}{k} -3}x\otimes x\right), &\text{ if } |x| < 1,\\
 \frac{1+k}{k}\Id_n, &\text{ if }|x| > 1.
\end{cases}
$$
Since the value $D^2f_k$ does not depend on $x$ for $|x| > 1$, we have that the matrix $A_k\doteq\cof D^2f_k$ defines a sequence of periodic, divergence-free and positive definite matrices in $X_p$. This in particular proves $(i)$. To show the strong convergence of $A_k$ in $L^p$ we note that 
$$
D^2f_k(x)\rightarrow
\begin{cases}
\frac{\Id_n }{|x|}-\frac{x\otimes x}{|x|^3}, &\text{ if } |x| < 1,\\
 \Id_n, &\text{ if } |x| > 1.
\end{cases}
$$
for almost every $x\in \T^n$. Moreover we can bound, uniformly in $k$, $|A_k(x)|\lesssim |D^2f_k|^{n-1} \leq |x|^{-(n-1)}\in L^p(\T^n)$, from which, by the dominated convergence theorem we conclude the validity of $(ii)$ where 
$$
A(x)=
\begin{cases}
\cof \left( \frac{\Id_n }{|x|}-\frac{x\otimes x}{|x|^3}\right), &\text{ if } |x|< 1,\\
 \Id_n, &\text{ if } |x| > 1.
\end{cases}
$$
We are now only left to show $(iii)$. We start by recalling that 
\begin{equation}\label{matrix_determinant_lemma}
\det (B+C)=\det B +\langle C, \cof^T B\rangle,
\end{equation}
for every $B,C\in \R^{n\times n}$, $\rank C=1$. This gives that 
\begin{align*}
\text{det}^{\frac{1}{n-1}} \cof  \left( \frac{\Id_n }{|x|}-\frac{x\otimes x}{|x|^3}\right)=\det \left( \frac{\Id_n }{|x|}-\frac{x\otimes x}{|x|^3}\right)=\frac{1}{|x|^n}-\left( \frac{x\otimes x}{|x|^3}, \frac{\Id_n}{|x|^{n-1}}\right)=0,\quad \text{for a.e.  } x,
\end{align*}
from which we can compute 
$$
\D(A)=\int_{\T^n} \text{det}^{\frac{1}{n-1}} A(x)\,dx=\int_{\T^n\setminus B_1}\det \Id_n\,dx=|\T^n\setminus B_1|.
$$
On the other hand, by using again \eqref{matrix_determinant_lemma},
\begin{align*}
\D(A_k)&=\int_{B_1}\det D^2f_k(x)\,dx +\int_{\T^n\setminus B_1} \det \left(  \frac{1+k}{k}\Id_n \right) \,dx\\
&= \left( \frac{1+k}{k} \right)^n \left[ \int_{B_1} \left( \det\left( \frac{\Id_n}{|x|^{1-\frac{1}{k}}}\right)+\left(\frac{1}{k}-1 \right)\frac{1}{|x|^{n\left(1-\frac{1}{k}\right)}} \right) \,dx+|\T^n\setminus B_1| \right]\\
&=\left( \frac{1+k}{k} \right)^n\left[\frac{1}{k}\int_{B_1}\frac{1}{|x|^{n\left(1-\frac{1}{k}\right)}}\,dx  +|\T^n\setminus B_1|\right]\\
&=\left( \frac{1+k}{k} \right)^n\Big(|B_1|+|\T^n\setminus B_1| \Big) =\left( \frac{1+k}{k} \right)^n|\T^n| .
\end{align*}
By letting $k\rightarrow\infty$ we conclude that 
$$
|\T^n|=\lim_{k\rightarrow \infty}\D(A_k)=\limsup_{k\rightarrow\infty}\D(A_k) >\D(A)=|\T^n\setminus B_1|.
$$
\end{proof}

\begin{rem}\label{exp}
The counterexample of Proposition \ref{counter} shows that one cannot hope to have $\{{\det}^\frac{1}{n-1}(A_k)\}_k$  bounded in $\mathcal{H}^1(\T^n)$  if $p < \frac{n}{n- 1}$, even if the divergence of the matrix field $A \in X_p$ is zero. Indeed, if such an estimate were true, then considering again the sequence $\{A_k\}_k$ constructed in Proposition \ref{counter}, we would find that the sequence $\{{\det}^\frac{1}{n-1}(A_k)\}_k$ is equi-integrable. Since $A_k$ converges to $A$ pointiwse a.e., by Vitali's convergence theorem we would then find $\lim_k\D(A_k) = \D(A)$, which cannot be true in view of  Proposition \ref{counter}-$(iii)$. 
\end{rem}

\begin{rem}\label{rem_hardy_crit}
The counterexample to the upper semicontinuity in the critical case $p=\frac{n}{n-1}$ we gave in \cite{USC}*{Prop. 8} shows that the Hardy regularity of ${\det}^\frac{1}{n-1}(A)$ from Theorem \ref{intro:det_hardy} cannot hold without the additional assumption \eqref{gencond} on the divergence.  Indeed, letting $\{A_k\}_k$ the $L^\frac{n}{n-1}(\T^n)$ bounded sequence of \cite{USC}*{Prop. 8}, and by reasoning exactly as in the above remark, we deduce that it is not possible to expect a uniform estimate of $\{{\det}^\frac{1}{n-1}(A_k)\}_k$ in $\mathcal{H}^1(\T^n)$, since it would imply continuity of $\D$ by Vitali's convergence theorem.  Note that the sequence $\{A_k\}_k$ of \cite{USC}*{Prop. 8} displays a Dirac Delta in the weak limit of $\{|\dv A_k|\}_k$,  which in particular shows that the assumption \eqref{gencond} is essentially sharp, since measures $\nu$ with an isolated atom fail to satisfy  $\tilde M \nu\in L^\frac{n}{n-1}(\T^n)$. See also the sufficient condition \eqref{mus_ball} given below.
\end{rem}

\subsection{Proof of Corollary \ref{intro:hardystrong}}

By Theorem \ref{intro:det_hardy} we get that the sequence $\left\{\det^{\frac{1}{n-1}}(A_k)\right\}_{k\in \N}$ is bounded in $\mathcal{H}^1(\T^n)$. Since $\{A_k\}_{k\in \N}\subset  C_\lambda$ we deduce that $f_k(x)\doteq |A_k(x)|^\frac{n}{n-1}$ is  bounded in $\mathcal{H}^1(\T^n)$ too. It follows that $\{f_k\}_{k}$ is a sequence of equi-integrable functions. This information, combined with the pointwise convergence yields by Vitali's convergence theorem the strong convergence of $A_k$ to $A$ in $L^\frac{n}{n-1}(\T^n)$. \qed

\appendix

\section{Radon Measures}\label{AppA}

The following is taken from \cite[Thm. 5.8]{MAG} and (the proof of) \cite[Thm. 2.17]{PMA}.

\begin{Teo}\label{RNdec}
Given two Radon measures $\alpha$, $\beta$ on $\R^n$, there exists a decomposition $\alpha = fd\beta + \alpha^s$, where $\alpha^s \perp \beta$ and 
\begin{itemize}
\item $f \in L^1(\R^n;\beta)$, and for $\beta$ a.e. $x \in \R^n$,
\[
f(x) = \lim_{r \to 0^+}\frac{\alpha(\overline{B_r(x)})}{\beta(\overline{B_r(x)})};
\]
\item $\alpha^s$ is concentrated on the set 
\[
E = \{x \in \R^n: x \notin \spt(\beta)\}\cup \left\{x \in \spt(\beta): \liminf_{r \to 0^+}\frac{\alpha(\overline{B_r(x)})}{\beta(\overline{B_r(x)})} = + \infty\right\}.
\]
\end{itemize}
\end{Teo}

We will also need the following result, that we took from \cite[Lemma 4.11]{DLDPHM} which states that every Radon measure $\mu$ is asymptotically doubling in a weak sense $\mu$-a.e..

\begin{prop}\label{asdoub}
Let $\mu$ be a Radon measure on $\R^n$ and set
\[
A_{\mu} \doteq \left\{x \in \spt(\mu): \limsup_{r \to 0^+}\frac{\mu(B_r(x))}{\mu(B_{2r}(x))} \ge 2^{-n}\right\}.
\]
Then $\mu(\R^n \setminus A_\mu) = 0$.
\end{prop}

\section{A sufficient condition for the validity of \eqref{gencond} in Theorem \ref{intro:det_hardy}}\label{AppB}

Recall from \eqref{gencond} that, for any positive Radon measure $\nu$ on $\T^n$, we write
\[
\tilde M \nu (x)\doteq\sup_{0 <R <1}\frac{1}{R^{n-1}}\nu(B_R(x)\cap \T^n).
\]

\begin{prop}\label{p:tildeMdivA}
Let $\nu\in \mathcal{M}^+(\T^n)$ be a positive Radon measure. Suppose that $\nu=g\,dx+\nu^s$, where $g\in \mathcal{H}^1(\T^n)$, and $\nu^s\in \mathcal{M}^+(\T^n)$  satisfies 
\begin{equation}\label{mus_ball}
\nu^s(B_r(x))\leq N r^\delta,
\end{equation}
for some $\delta>0$, $\forall r\in (0,1)$ and every $x\in \T^n$, for some universal constant $N>0$. Then
\begin{equation}\label{est:div_in_critical}
\|\tilde M\nu\|_{L^\frac{n}{n - 1}(\T^n)}\leq C_n\|g\|_{\mathcal{H}^1(\T^n)}+ C_{n,\delta}\left(\|\nu^s\|_{\mathcal{M}^+(\T^n)}+N \right).
\end{equation}
\end{prop}

\begin{proof}
We will prove separately that 
\begin{align}
\| \tilde Mg\|_{L^\frac{n}{n - 1}(\T^n)}&\leq C_n\|g\|_{\mathcal{H}^1(\T^n)}\label{first}\\
\|\tilde M\nu^s\|_{L^\frac{n}{n - 1}(\T^n)}&\leq C_{n,\delta}\left(\|\nu^s\|_{\mathcal{M}^+(\T^n)}+N \right)\label{second}.
\end{align}
We will denote by $I_1$ and $\mathcal{R}$ the Riesz potential and the vector-valued Riesz transform respectively.  More precisely, for every function $h\in L^1(\T^n)$, we set 
\begin{align*}
I_1 h(x)=\int_{\T^n} \frac{h(y)}{|x-y|^{n-1}}\, dy \quad \text{ and } \quad \mathcal{R} h=p.v. \int_{\T^n} \frac{x-y}{|x-y|^{n+1}} h(y)\, dy,
\end{align*}
with the usual straightforward modifications when $h$ is substituted by a measure $\nu$. Note that
$$
\tilde M g(x)=\sup_{R>0}\frac{1}{R^{n-1}}\int_{B_R(x)} g(y)\,dy\leq \sup_{R>0}\int_{B_R(x)} \frac{g(y)}{|x-y|^{n-1}}\,dy=\int_{\T^n} \frac{g(y)}{|x-y|^{n-1}}\,dy=I_1 g(x),
$$
and similarly $\tilde M \nu^s\leq I_1\nu^s$. Thus to prove \eqref{first} and \eqref{second} it will be enough to estimate $I_1 g$ and $I_1\nu^s$. 
By the Hardy--Littlewood--Sobolev inequality we have 
\begin{equation}
\label{Rieszinhardy}
\|I_1 g\|_{L^\frac{n}{n - 1}(\T^n)}\leq C\|\mathcal{R}g\|_{L^1(\T^n)}\leq C \|g\|_{\mathcal{H}^1(\T^n)},
\end{equation}
where in the second inequality we used that $\|g\|_{\mathcal{H}^1}\doteq \|g\|_{L^1}+\|\mathcal{R} g\|_{L^1}$ defines the norm in the Hardy space, see \cite{FEF}.  This directly gives \eqref{first}. Estimate \eqref{second} is well-known, and can be found, for instance, in the proof of \cite[Prop. 6.1]{Sil05}.
\end{proof}

\bibliographystyle{plain}
\bibliography{USC}

\end{document}